\documentclass[10pt,a4paper,twoside]{amsart}
\usepackage[T1]{fontenc}
\usepackage[utf8]{inputenc}
\usepackage[english]{babel}
\usepackage{amsmath}
\usepackage{amsthm}
\usepackage{amssymb}
\usepackage{amsfonts}
\usepackage{amsxtra}
\usepackage{enumerate}
\usepackage{verbatim}
\usepackage{color}
\usepackage[mathscr]{eucal}
\usepackage{graphicx}
\usepackage{xcolor}
\usepackage[backref=page]{hyperref}
\usepackage{tikz}
\usepackage{enumitem}
\usepackage{parskip}

\newcommand{\Z}{\mathbb Z}

\newcommand{\R}{\mathbb R}
\newcommand{\C}{\mathbb C}

\newcommand{\om}{\omega}
\newcommand{\al}{\alpha}

\newtheorem*{theo}{Theorem}

\newtheorem*{nota}{Notation}
\newtheorem{theorem}{Theorem}[section]

\newtheorem{lemma}[theorem]{Lemma}
\newtheorem{cor}[theorem]{Corollary}

\newtheorem{prop}[theorem]{Proposition}

\newtheorem{remark}[theorem]{Remark}
\newtheorem{qu}[theorem]{Question}
\theoremstyle{definition}

\let\phi=\varphi

\newcommand{\rem}[1]{}

\DeclareFontFamily{U}{mathb}{\hyphenchar\font45}
\DeclareFontShape{U}{mathb}{m}{n}{
<-6> mathb5 <6-7> mathb6 <7-8> mathb7
<8-9> mathb8 <9-10> mathb9
<10-12> mathb10 <12-> mathb12
}{}
\DeclareSymbolFont{mathb}{U}{mathb}{m}{n}
\DeclareMathSymbol{\llcurly}{\mathrel}{mathb}{"CE}
\DeclareMathSymbol{\ggcurly}{\mathrel}{mathb}{"CF}

\newcommand{\id}{\mathrm{id}}

\title[Non-orderability and the contact Hofer norm]{Non-orderability and the contact Hofer norm}

\author{Jakob Hedicke}
\address{ Centre de recherches mathématiques (CRM), Université de Montréal, 2920 Chemin de la tour, Montréal (Québec), H3T 1J4, Canada} 
\email{jakob.hedicke@gmail.com}

\author{Egor Shelukhin}
\address{D\'epartement de Math\'ematiques et de Statistique, Universit\'e de
Montr\'eal, C.P. 6128 Succ. Centre-Ville, Montreal (Québec), H3C 3J7, Canada} 
\email{egor.shelukhin@umontreal.ca}

\date{\today}
\begin{document}

\begin{abstract}
We relate non-orderability in contact topology to shortening in the contact Hofer norm. Combined with considerations of open books, this provides many new examples of non-orderable contact manifolds, including contact boundaries of subcritical Weinstein domains, and in particular the long-standing case of the standard $S^1 \times S^2.$ We also produce new examples of contact manifolds admitting contactomorphisms without translated points, provide obstructions to subcritical polarizations of symplectic manifolds, and establish a $\mathcal{C}^0$-continuity property of the contact Hofer metric.
\end{abstract} 

\maketitle

\tableofcontents

\section{Introduction}

Let $(M,\xi)$ be a cooriented contact manifold.
Throughout this paper, we denote by $\mathcal{G}=\mathcal{G}(M,\xi)=\mathrm{Cont}_0(M,\xi)$ the identity component of the group of compactly supported contactomorphisms, that is, the group of diffeomorphisms preserving $\xi$ that are connected to the identity by a smooth compactly supported isotopy of such diffeomorphisms. We denote by $\tilde{\mathcal{G}}=\Tilde{\mathcal{G}}(M,\xi)$ the universal cover of $\mathcal{G}$.

In \cite{Eliashberg00} Eliashberg and Polterovich have introduced a natural bi-invariant relation $\preccurlyeq$ on the group $\Tilde{\mathcal{G}}$ by looking at non-negative paths of contactomorphisms, i.e., paths that are positively transverse to the cooriented contact structure.
In case $M$ is closed, they showed that the relation is a partial order, or in other words that $M$ is {\bf orderable}, if and only if there does not exist a contractible positive loop of contactomorphisms and gave various examples of contact manifolds without contractible positive loops. This has since been established in many further cases (see e.g. \cite{Albers15, Albers23, Bae24, Cant23, Colin19, Chernov102, Eliashberg06, Granja21}). However, so far, all known examples of non-orderable contact manifolds arise as the boundaries of $2$-stabilizations of Liouville manifolds, see \cite{Eliashberg06}.

In an a priori different direction, by analogy to the Hofer norm in symplectic geometry, Shelukhin introduced the contact Hofer norm $|\cdot|_{\alpha}$ on $\mathcal{G}$ and a pseudo-norm on $\tilde{\mathcal{G}}$, that depends on the choice of a contact form $\alpha$, see \cite{Egor}. A first connection between orderability and the Hofer norm was found in \cite{Hedicke24}, where it was shown that if there are no positive loops of contactomorphisms, then the diameter of $\mathcal{G}$ with respect to $|\cdot|_{\alpha}$ is infinite. The starting point of this paper is a generalization of this result to $\Tilde{\mathcal{G}}(M)$. It allows us to replace the question of existence of positive loops by the question of estimates on the contact Hofer norm.

We proceed by giving several examples of contact manifolds for which the contact Hofer diameter is bounded on $\mathcal{G}$ and $\Tilde{\mathcal{G}}$. This gives a conceptual criterion for the existence of (contractible) positive loops. In particular we prove the following by means of contact open books and Legendrian flexibility. We refer to Section \ref{sec: sketch} for an outline of our arguments.

\begin{theorem}\label{thm1}
Let $(M^{2n+1},\xi)$ be a closed contact manifold, $W$ a page of a Weinstein open book supporting $\xi$ and $L$ the skeleton of $W$ with respect to some ideal Giroux form.
Then the following holds.
\begin{enumerate}[label=(\roman*)]
    \item\label{thm1:subcrit} If $W$ is a subcritical Weinstein manifold, the contact Hofer norms are bounded on $\mathcal{G}$ and $\Tilde{\mathcal{G}}$.
    \item\label{thm1:loose} If $n\geq 2$ and $L$ is a (smooth and closed) loose Legendrian in $M$, then for any contact form $\alpha$ the restriction of the contact Hofer norm to the group of strict contactomorphisms $\mathcal{G}_{\alpha}$ and its pre-image in the universal cover $\Tilde{\mathcal{G}}_{\alpha}$ is bounded \footnote{Here $\Tilde{\mathcal{G}}_{\alpha}\subset \Tilde{\mathcal{G}}$ denotes the set of all classes in $\Tilde{\mathcal{G}}$ that can be represented by a path $\phi_t$ with $\phi_1\in\mathcal{G}_{\alpha}$.}.
\end{enumerate}

In both cases $(M,\xi)$ is non-orderable.
\end{theorem}

\begin{remark}
  Note that Theorem \ref{thm1} does not cover any overtwisted manifolds.
  As pointed out in \cite{Borman15}, any loose Legendrian in an overtwisted contact manifold has an overtwisted complement.
  On the other hand the complement of the skeleton of a page contracts to an arbitrarily small neighbourhood of any other skeleton, see Lemma \ref{lem_ob_contraction}, and therefore cannot be overtwisted. In particular, for an overtwisted contact manifold, there does not exist a supporting open book with flexible pages. This observation is related to \cite[Question 1.1]{BowdenCrowley}.
\end{remark}

Theorem \ref{thm1} in particular implies the following, partially answering a question raised in \cite{Eliashberg06}.

\begin{cor}\label{cor1stab}
Let $W$ be a Weinstein manifold with a smooth closed skeleton.
Then the ideal boundary of $W\times \mathbb{C}$ is non-orderable.
\end{cor}

\begin{remark} A few remarks are in order.
\begin{enumerate}[label=(\roman*)]
    \item \label{rem:cieliebak} Due to a result of Cieliebak \cite[Theorem 14.16]{Cieliebak12}, the Weinstein manifolds of the form $W\times \mathbb{C}$ are the subcritical Weinstein manifolds.
    \item If $W$ is critical, i.e., the skeleton is a closed Lagrangian $L$, $W$ is exact symplectomorphic to $T^{\ast}L$, see \cite[Theorem 1.1]{Starkston18} and \cite[Corollary 11.21]{Cieliebak12}.
    \item Corollary \ref{cor1stab} also holds when $\dim(W)=2$. In this case the only critical Weinstein manifold with smooth closed skeleton is $T^{\ast}S^1$. The ideal boundary of its $1$-stabilization is the standard tight $S^1\times S^2$, whose non-orderability is covered by Corollary \ref{corst} using a different approach: namely, the more ad-hoc Theorem \ref{thm2} below.
\end{enumerate}
\end{remark}

We can strengthen Corollary \ref{cor1stab} as an application of Theorem \ref{thm:perreeb} below to cover ideal contact boundaries of all subcritical Weinstein manifolds $W \times \C$ of dimension at least $6,$ answering a question raised in \cite{Eliashberg06}.

\begin{theorem}\label{thm1stab}
Let $W$ be a Weinstein manifold of finite type of dimension at least $4.$ Then the ideal contact boundary of $W\times \mathbb{C}$ is non-orderable.
\end{theorem}

\begin{remark}\label{rem:2d} A few further remarks.
\begin{enumerate}[label=(\roman*)]
    \item\label{rem:2d1}  We expect that Theorem \ref{thm1stab} holds also for $\dim(W)=2.$ Indeed, one can see from \cite{Casals19} that in this case the skeleton of the page is a full stabilization of itself, a case where an analogue of \cite{Nakamura21} is expected to apply (see \cite{DelpinoNakamura}). 
    \item We also expect to generalize Theorem \ref{thm1} to the case of (possibly singular) loose isotropic complexes $L$ and in particular, to generalize part (i) to the case of flexible pages $W$ (with arbitrary monodromy).
    \end{enumerate} \end{remark}

A different class of examples satisfying the conditions of Theorem \ref{thm1} is that of prequantizations of symplectic manifolds, whose induced complex line bundles provide subcritical polarizations of the symplectic form. This means that they admit sections whose zero-divisors are symplectic codimension two submanifolds and have subcritical Weinstein complements, see \cite{Opsh13}. We refer to \cite{Chiang14} for the constructions of the relevant open books and to \cite{BiranCieliebak01} for information concerning subcritical polarizations. 

\begin{cor}\label{cor:Kahler}
Let $(P,\alpha)$ be the prequantization of a symplectic manifold corresponding to an Hermitian line bundle providing a subcritical polarization. Then $(P,\xi)$ for $\xi = \ker(\alpha)$ is non-orderable.   
\end{cor}

See Section \ref{sec:prequ} for a discussion and generalization of Corollary \ref{cor:Kahler}.

\begin{remark}\label{rem:frol} As brought to our attention by Frol Zapolsky, the non-orderability of this class of examples follows formally from results in \cite{Courte18} and \cite{Fraser}. Our argument can more easily be made effective, in all cases, to produce explicit positive contractible loops: see Section \ref{sec:explicit}. 
\end{remark}

\begin{remark}
We expect the results in \cite{Courte18} to generalize verbatim to the setting where the page of the open book is loose inside the contact manifold, and in particular to ideal contact boundaries of subcritical Weinstein manifolds (see Section \ref{sec:proof subcrit}). 
\end{remark}

Note that Corollary \ref{cor:Kahler} implies that if the prequantization $(P,\alpha)$ is orderable (see \cite{Albers23, Bae24} for examples), then its base does not admit a subcritical polarization.
We refer to Section \ref{sec:transl} for further applications of our methods to the existence of contactomorphisms without translated points, which provide further obstructions to the existence of subcritical polarizations.

Using similar techniques as for Theorem \ref{thm1} we prove the following.

\begin{theorem}\label{thm2}
Let $(M,\xi)$ be a closed contact manifold, $L\subset M$ be the skeleton of a page of an open book decomposition supporting $\xi$ and $\tilde{\phi}_t$ be the lift of a path $\phi_t$ in $ \mathcal{G}$ to $\tilde{\mathcal{G}}$.
If there exists a contactomorphism $\psi$ such that $\psi(L)\cap \phi_t(L)=\emptyset$ for all $t\geq 0$, then the contact Hofer norm is bounded along $\tilde{\phi}_t$ and hence along $\phi_t$.
In particular, if $\tilde{\phi}_t$ is induced by the Reeb flow of some contact form, then $(M,\xi)$ is not orderable.
\end{theorem}

See Section \ref{secfrag} for the proofs of Theorem \ref{thm1} and Theorem \ref{thm2}.

\begin{remark}
In fact the proof of Theorem \ref{thm2} shows the following:
Let $(K_1,\theta_1)$ and $(K_2,\theta_2)$ be two (possibly different) open book decompositions supporting $(M,\xi)$.
If there exist a page $W_1$ of $(K_1,\theta_1)$ with skeleton $L_1$, a page $W_2$ of $(K_2,\theta_2)$ with skeleton $L_2$ and a Reeb flow $\phi_t^{\alpha}$ such that $\phi_t^{\alpha}(L_1)\cap L_2=\emptyset$, then $(M,\xi)$ is not orderable.
\end{remark}

In many cases the Reeb flow $\phi_t^{\alpha}$ can be chosen such that the image of $L$ under the Reeb flow is a pre-Lagrangian submanifold in $M$.
Using results on displacement of pre-Lagrangian tori, see \cite{Marincovic16}, it follows:

\begin{cor}\label{corst}
For any $n\geq 1$ the manifold $T^n\times S^{n+1}$, equipped with its standard contact structure as the ideal contact boundary of the $1$-stabilization of $T^{\ast}T^n$, is not orderable.
\end{cor}

Furthermore, Theorem \ref{thm2} allows us to recover all of the examples of non-orderable manifolds due to Eliashberg-Kim-Polterovich \cite{Eliashberg06} as follows.

\begin{cor}\label{corekp}
Let $W$ be a Liouville manifold. Then the ideal contact boundary of the $2$-stabilization $W \times \C^2$ is non-orderable.
\end{cor}

Part (ii) of Theorem \ref{thm1} and Corollary \ref{corst} do not show that the contact Hofer norm is bounded on $\mathcal{G}(T^n\times S^{n+1})$ or its universal cover. This raises the following questions.

\begin{qu} 
1. Are there non-orderable contact manifolds that have an unbounded contact Hofer norm on $\tilde{\mathcal{G}}$?\\
2. Are there contact manifolds with a positive loop of contactomorphisms and unbounded contact Hofer norm on $\mathcal{G}$?
\end{qu}
We believe that the answer to the second question is positive. This will be investigated in \cite{ASZ}. Georgios Dimitroglou-Rizell suggested that the standard $\R P^{3}$ might be an example with bounded contact Hofer norm on $\mathcal{G},$ while on $\tilde{\mathcal{G}}$ it is known to be unbounded by Givental's quasimorphism \cite{Giv}. See Section \ref{sec:almostReeb} for further questions of this kind.

\begin{remark}
We note that in all cases where we proved non-orderability, by \cite[Section 1.7]{Eliashberg06}, we obtain squeezing results for domains inside $(SM \times S^1, \ker(\lambda-dt))$ where $SM = \R \times M$ is the symplectization and $\lambda = e^\theta \alpha$ for $\theta$ the coordinate on $\R,$ or inside $W \times S^1$ in case where $M = \partial_{\infty} W$ for a Liouville domain $W.$
\end{remark}

Finally, we formulate a rather ad-hoc generalization of Theorem \ref{thm1} \ref{thm1:loose} in the spirit of Theorem \ref{thm2}, which implies Theorem \ref{thm1stab}. Recall that an isotropic complex, see \cite[Definitions 3.1, 3.4]{Courte18}, is said to be {\bf loose} if its $n$-dimensional stratum is loose in the complement of its $(n-1)$-dimensional skeleton.

\begin{theorem}\label{thm:perreeb}
Let $(M,\xi)$ be a closed contact manifold, $L\subset M$ be the skeleton of a page of a Weinstein open book decomposition supporting $\xi$ and $\alpha$ be a contact form whose Reeb flow $\phi_t^{\alpha}$ is periodic near the subcritical part $L^{n-1}.$ If $L$ is a loose isotropic complex, then the contact Hofer norm of the lift $\tilde{\phi}_t^{\alpha}$ to $\widetilde{\mathcal{G}}$ is bounded as a function of $t.$ In particular $(M,\xi)$ is not orderable. \end{theorem}

\begin{remark}
We remark that our use of the $h$-principle for loose Legendrian complexes is different from that in \cite{Courte18}. Indeed, we do not use it to remove intersections with the skeleton of a page, but instead to shorten the contact Hofer norm by \cite{Nakamura21}.    
\end{remark}

Finally, we comment on Legendrian versions of our results, which have been more extensively studied in the literature. Namely, analogously to the Chekanov Hofer distance for Lagrangian submanifolds, it is possible to define a contact Hofer distance for Legendrian isotopy classes, see \cite{Rosen}. The results relating orderability of Legendrian isotopy classes to the contact Hofer diameter also hold in the Legendrian case \cite{Hedicke24}. Let $\mathcal{L}$ be an isotopy class of a closed Legendrian $\Lambda.$ Nakamura \cite{Nakamura21} proves that the contact Hofer diameter of the space of Reeb flow images $\{ \phi_t^{\alpha}(\Lambda)\} \subset \mathcal{L}$ of every closed loose Legendrian $\Lambda$ in a closed contact manifold is bounded. This implies immediately that there exists a positive loop of Legendrians based at $\Lambda,$ partially recovering a result of Liu \cite{Liu20}. We expect that the techniques \cite{Nakamura21} should extend to an analogous statement for the universal cover $\widetilde{\mathcal{L}},$ which would completely recover Liu's result.

\subsection{Outline of the proofs}\label{sec: sketch}

The arguments proving the above results combine the following main ingredients. 

First, Lemma \ref{lem_minimum} implies that non-orderability is equivalent to the shortening of the Reeb flow $\widetilde{\phi}_t^\al$ in $\widetilde{\mathcal G}$ with respect to the contact Hofer pseudo-norm associated with $\al.$ Namely, the existence of $T>0$ such that $|\widetilde{\phi}_T^\al|_{\al}<T.$ We note that this can be carried out for {\it any} contact form $\al$ as non-orderability is an invariant notion.

Second, we prove various instances of such shortening by considering properties of open book decompositions of contact manifolds. Notably, in Proposition \ref{prophandle} we prove that the contact Hofer norm is uniformly bounded on the image $\widetilde{\mathcal{G}}(M \setminus L) \to \widetilde{\mathcal{G}}(M),$ where $L$ is the skeleton of one page of an open book on $M.$ We call such $M \setminus L$ skeleton-complements. Note that if the page is Weinstein, the skeleton is an isotropic complex, while if it is Liouville, this is not generally the case. In this step we use the existence of a contact flow contracting the complement of one skeleton into an arbitrarily small neighbourhood of a second one (see Lemma \ref{lem_ob_contraction} or \cite{Courte18}).

This reduces the shortening of $\widetilde{\phi}_t^\al$ essentially to considering the isotopy $\widetilde{\phi}_t^\al(L)$ of isotropic complexes. For instance, if this isotopy is disjoint from a skeleton $L'$ of another page or in fat from the image a page $\psi(L')$ under a contactomorphism, then using Lemma \ref{lemfragmentation} about fragmentation, we reduce the isotopy $\widetilde{\phi}_t^\al$ to the product of two isotopies supported in skeleton-complements and hence is bounded in the contact Hofer norm. This proves Theorem \ref{thm2} which implies Corollaries \ref{corekp} and \ref{corst}. In particular, this yields the non-orderability of the standard $S^1 \times S^2.$

Third, if the skeleton $L$ of a page is a loose Legendrian as defined in \cite{Murphy12} then following \cite[Theorem 1.2]{Nakamura21} we can produce an isotopy $(f_t)_{0 \leq t \leq T}$ of uniformly bounded $\al$ Hofer length which brings each $\phi^{\al}_t(L)$ for $0 \leq t \leq T$ to a $C^0$-small neighbourhood of $L$ (here it is important that the $\phi_t^{\al}$ are strict contactomorphisms, i.e. they preserve $\al$). In particular, by fragmentation again, $[f_t \phi_t^{\al}] \in \widetilde{\mathcal{G}}$ is of uniformly bounded $\al$ contact Hofer norm. This shows that $|\widetilde{\phi}_T^{\al}|_{\al}$ is uniformly bounded in $T,$ which yields Theorem \ref{thm1:loose}. Furthermore, if $W$ is a Weinstein manifold with a smooth skeleton $L$, then $L$ considered as a Legendrian submanifold in the ideal boundary of $W \times \C$ is loose due to \cite{Casals19}, which yields Corollary \ref{cor1stab}.

Finally, to prove Theorem \ref{thm1stab} we combine the approaches described in the two preceding paragraphs: the case of a subcritical skeleton, and the case of a smooth loose skeleton. Indeed, we first modify the Reeb isotopy to fix the subcritical part of the skeleton, and then proceed by a relative version of the loose arguments. To prove the necessary uniform estimates, we assume the Reeb flow to be periodic on the subcritical part of the skeleton. This produces Theorem \ref{thm:perreeb}. Its conditions are satisfied for contact boundaries of subcritical domains by another application of \cite{Casals19}.

\subsection{Future directions} 
It would be interesting to make further progress on the question of orderability of the ideal contact boundaries of Weinstein manifolds beyond Theorem \ref{thm1stab}. For instance, the properties of flexible Weinstein manifolds often mirror those of subcritical ones, cf. \cite{Lazarev}. Therefore, in view of Theorem \ref{thm1stab}, we expect that ideal contact boundaries of flexible Weinstein manifolds should be non-orderable. We plan to investigate this question in future work.

Similarly, one may wonder, following \cite[Section 1.9]{Eliashberg00}, whether overtwisted contact manifolds are non-orderable. There currently exists no known example of this phenomenon, as we define it, in the literature. We refer to Section \ref{sec:ord} for further discussion of the state of the art. We hope to apply the $h$-principles of \cite{Borman15} to this question.

Finally, it is suggested by our results that there should be a close relation between non-orderability and the existence of contactomorphisms without translated points, to be investigated in the future.

\section*{Acknowledgments}
We thank Georgios Dimitroglou Rizell for suggesting proving Theorem \ref{thm1stab} and for very helpful related discussions. We thank Yasha Eliashberg for an encouraging discussion and mentioning squeezing in the context of this paper. We thank Peter Albers for suggesting an improvement to Theorem \ref{thm2}, Paul Biran for useful communications about subcritical polarizations, Lukas Nakamura for useful communications regarding Remark \ref{rem:2d}(i), Frol Zapolsky for contributing Remark \ref{rem:frol} and Jun Zhang for his contribution to Section \ref{sec:almostReeb}. 
We also thank the organizers of the 2024 Conference ``Contact and Lorentzian geometry II" in Bochum, as well as the organizers of the seminars in Heidelberg University and the University of Maryland for a possibility to present preliminary versions of our results.

J.H. was supported by the CRM-ISM and CIRGET postdoctoral fellowships and by the Fondation Courtois. E.S. was supported by an NSERC Discovery Grant, the Fondation Courtois, and a Sloan Research Fellowship.

\section{Background}

\subsection{Orderability}\label{sec:ord}

Recall that $\tilde{\mathcal{G}}$ can be considered as the set of homotopy classes of paths in $\mathcal{G}$ starting at the identity.
\begin{nota}
Given a path $(\phi_s)_{s\in I}$ in $\mathcal{G}$ starting at the identity, we denote its lift to $\tilde{\mathcal{G}}$ at time $t\in I$ by $\tilde{\phi}_t$.
\end{nota}

Following \cite{Eliashberg00} we define a bi-invariant relation on the groups $\mathcal{G}$ and $\Tilde{\mathcal{G}}$ for any closed cooriented contact manifold $(M,\xi=\ker\alpha)$ as follows.

Given $\tilde{\phi},\tilde{\psi}\in\Tilde{\mathcal{G}}$ we say that $\tilde{\phi}\preccurlyeq \tilde{\psi}$ if and only if $\tilde{\phi}^{-1}\tilde{\psi}$ can be represented by a non-negative path, i.e., if and only if there exists a path $(f_t)_{t\in[0,1]}$ in $\mathcal{G}$ representing $\tilde{\phi}^{-1}\tilde{\psi}$ such that $\alpha\left(\frac{d}{dt}f_t\right)\geq 0$ for all $t$.

A similar relation can be defined on $\mathcal{G}$ by setting $\phi\preccurlyeq \psi$ if and only if there exists a non-negative path connecting $\mathrm{id}$ and $\phi^{-1}\psi$.

As in \cite{Eliashberg00} we call the contact manifold $(M,\xi)$ \textbf{orderable} if the relation $\preccurlyeq$ is a partial order on $\Tilde{\mathcal{G}}$.
Otherwise, we call the contact manifold \textbf{non-orderable}.
A closed contact manifold is orderable if and only if there are no contractible positive loops in $\mathcal{G}$ and non-orderable iff the relation $\preccurlyeq$ is trivial.

Moreover, following \cite{Chernov16} we call $\mathcal{G}$ orderable if the relation $\preccurlyeq$ defines a partial order on $\mathcal{G}$, which is equivalent to the non-existence of any positive loops in $\mathcal{G}$.

Orderability is known to hold in several cases, such as certain prequantization spaces \cite{Eliashberg00,Borman152}, spherical cotangent bundles \cite{Chernov102} and more generally any contact manifold that admits an atoroidal Liouville at infinity filling with non-vanishing symplectic cohomology \cite{Colin19, Cant23}.
The orderability of $\mathcal{G}$ seems to be more restrictive, as it prevents for example the existence of periodic Reeb flows. 
This holds, for instance, for a class of spherical cotangent bundles of manifolds that admit an open cover; see \cite{Chernov102}.

On the other hand, very little seems to be known about non-orderable contact manifolds.
So far, all examples of contractible positive loops arise from the construction in \cite{Eliashberg06}, where the authors construct contractible positive loops on contact manifolds that are the ideal boundary of the $2$-stabilization of some Liouville manifold.

In the $3$-dimensional case non-orderability of $\mathcal{G}$ was proved in \cite{Casals16} for the standard contact structure on $S^2\times S^1$ and a certain class of overtwisted manifolds using fibre connected sums around orbits of a given positive loop. The case of $\widetilde{\mathcal{G}},$ however, remained open.

From the work Borman, Eliashberg and Murphy \cite{Borman15} and the work of Liu \cite{Liu20} it follows that overtwisted manifolds of any dimension are non-orderable in a slightly weaker sense.

A natural analogue of orderability can be defined on isotopy classes of closed Legendrian submanifolds and its universal cover, see e.g. \cite{Chernov102}.
Let $\mathcal{L}$ be the isotopy class of a closed Legendrian $\Lambda$, i.e., the orbit of $\Lambda$ under the action of $\mathcal{G}$.
We say that an isotopy of Legendrians $\lambda_t$ in $\mathcal{L}$ is positive (non-negative) if there exists a parametrisation $\iota_t\colon \Lambda\rightarrow\Lambda_t$ such that for one and hence any positive contact form $\alpha$ and all $p\in\Lambda$ we have
$$\alpha\left(\frac{d}{dt}\iota_t(p)\right)>0(\geq 0).$$
As in the case of contactomorphisms, this allows to define a relation on $\mathcal{L}$ and $\tilde{\mathcal{L}}$, which is a partial order iff there is no positive (contractible) loop of Legendrians \cite[Proposition 4.5]{Chernov102}.

\subsection{The contact Hofer norm}

Inspired by the Hofer norm for compactly supported Hamiltonian diffeomorphisms of a symplectic manifold, one can define a family of (pseudo-) norms on $\mathcal{G}$, see \cite{Egor}.
Unlike the Hofer norm, these norms are not conjugation invariant but depend on the choice of a contact form.

Let $(\phi_t)_{t\in [0,1]}$ be a path in $\mathcal{G}$ and $\alpha$ a contact form.
The \textbf{contact Hofer length} of can be defined as
$$l_{\alpha}((\phi_t)_{t\in [0,1]})=\int\limits_0^1\max\limits_M\left\vert\alpha\left(\frac{d}{dt}\phi_t\right)\right\vert dt.$$
Define the \textbf{contact Hofer norm} of a contactomorphism $\phi\in\mathcal{G}$ as
$$|\phi|_{\alpha}:=\inf l_{\alpha}((\phi_t)_{t\in [0,1]}),$$
where the infimum is taken over all paths $\phi_t$ with $\phi_0=\mathrm{id}$ and $\phi_1=\phi$.

Similarly for a homotopy class of paths $\tilde{\phi}\in\Tilde{\mathcal{G}}$ define $|\tilde{\phi}|_{\alpha}$ as the infimum of the contact Hofer lengths of paths in the class $\tilde{\phi}$.

In \cite[Theorem A]{Egor} the following properties of $|\cdot|_{\alpha}$ are proved:

\begin{theorem}[\cite{Egor}]
The contact Hofer norm on $\mathcal{G}$ satisfies
\begin{itemize}
    \item[(i)] $|\phi|_{\alpha}=0$ iff $\phi=\mathrm{id}$.
    \item[(ii)] $|\phi\psi|_{\alpha}\leq |\phi|_{\alpha}+|\psi|_{\alpha}$
    \item[(iii)] $|\phi^{-1}|_{\alpha}=|\phi|_{\alpha}$
    \item[(iv)] $|\psi\phi\psi^{-1}|_{\alpha}=|\phi|_{\psi^{\ast}\alpha}$. 
\end{itemize}
On $\Tilde{\mathcal{G}}$ the function $|\cdot|_{\alpha}$ defines a pseudo norm satisfying properties (ii)-(iv).
\end{theorem}

Assuming that there are no positive loops in $\mathcal{G}$ the following can be proved, see \cite[Corollary 2.6]{Hedicke24}.

\begin{theo}[\cite{Hedicke24}]
Assume that $\mathcal{G}$ is orderable. 
Let $\phi_t^{\alpha}$ be the Reeb flow of some contact form $\alpha$ at time $t\in\R$.
Then 
$$|\phi_t^{\alpha}|_{\alpha}=|t|.$$
\end{theo}

The result in particular implies the existence of a positive loop in $\mathcal{G}$ if the contact Hofer norm is bounded along the Reeb flow of some contact form.
This observation can be extended to $\tilde{\mathcal{G}}$ as follows.

\begin{theorem}\label{thm_hofer_orderable}
Let $\tilde{\phi}_t^{\alpha}$ be the lift of the Reeb flow at time $t$ of some contact form $\alpha$ to $\Tilde{\mathcal{G}}$.
If $\vert \tilde{\phi}_t^{\alpha}\vert_{\alpha}<|t|$ for some $t\in\mathbb{R}$ then $(M,\xi)$ is non-orderable, i.e., there exists a positive contractible loop in $\mathcal{G}$.
\end{theorem}

The key step to prove Theorem \ref{thm_hofer_orderable} is the following lemma adapted from \cite[Lemma 3.1]{Hedicke24}.
Note that the path constructed in the proof is homotopic to the original path, i.e., the Lemma naturally extends to $\Tilde{\mathcal{G}}$.

\begin{lemma}\label{lem_minimum}
For a path of contactomorphisms $(\phi_t)_{t\in[0,1]}$ set 
$$c := \int\limits_0^1\min\limits_M\alpha\left(\frac{d}{dt}\phi_t\right) dt.$$
For every $\delta>0$ there exists a path $(\psi_t)_{t\in[0,1]}\in \tilde{\phi}_t$ such that for all $t\in[0,1]$
$$\min\limits_M\alpha\left(\frac{d}{dt}\psi_t\right) \in (c-\delta,c+\delta).$$
\end{lemma}

\begin{proof}[Proof of Theorem \ref{thm_hofer_orderable}]
We show that if $(M,\xi)$ is orderable, $|\tilde{\phi}_t^{\alpha}|_{\alpha}=|t|$ for any $t\in\R$.

From the definitions it is clear that $|\tilde{\phi}_t^{\alpha}|_{\alpha}\leq |t|$.
Assume that $t>0$ and $|\tilde{\phi}_t^{\alpha}|_{\alpha}< t$.
Then, as $|\tilde{\phi}_t^{\alpha}|_{\alpha} = |(\tilde{\phi}_t^{\alpha})^{-1}|_{\alpha}$ there exists a path $(\psi_s)_{s\in[0,t]}\in \tilde{\phi}_t^{\alpha}$ such that 
$$\int\limits_0^t\min\limits_M\alpha\left(\frac{d}{dt}\psi_s\right) ds=c, $$
for $|c|<1$.
In particular $1+c>0.$
Using Lemma \ref{lem_minimum} we can assume that $\min\limits_M\alpha\left(\frac{d}{ds}\psi_s\right)$ is arbitrarily close to $c$ for each $s\in [0,t]$.
Then $(\phi_{s}^{\alpha}\psi_s)_{s\in[0,t]}$ defines a contractible loop with
$$\min\limits_M \alpha\left(\frac{d}{ds}\left(\phi_{s}^{\alpha}\psi_s\right)\right)=1+\min\limits_M \alpha\left(\frac{d}{ds}\left(\psi_s\right)\right)\geq 1+c-\delta > 0,$$ as we can pick $\delta>0$ arbitrarily small.
It follows that $(M,\xi)$ is not orderable.
\end{proof}

In \cite{Rosen} Rosen and Zhang introduced a contact Hofer pseudo-metric on contact isotopy classes of subsets. 
In the case of an isotopy class of a closed Legendrian, this pseudo-metric is nondegenerate, see \cite{Dimitroglou24}.
Given the isotopy class $\mathcal{L}$ of a closed Legendrian and a contact form $\alpha$, the metric $d_{\alpha}$ between $\Lambda_0,\Lambda_1\in\mathcal{L}$ can be defined as the infimum over the contact Hofer lengths of paths $\phi_t$ such that $\phi_0=\mathrm{id}$ and $\phi_1(\Lambda_0)=\Lambda_1$.
Clearly, this pseudo-metric lifts to $\tilde{\mathcal{L}}$.
Moreover, using \cite[Lemma 5.1]{Hedicke24} one can prove an analogue of Theorem \ref{thm_hofer_orderable}.

Let $\mathcal{L}$ be the Legendrian isotopy class of a closed Legendrian $\Lambda$.
We consider the universal cover $\tilde{\mathcal{L}}$ as homotopy classes of Legendrian isotopies starting at $\Lambda$ and denote by $\tilde{\Lambda}$ the class of the trivial isotopy.

\begin{theorem}\label{thm_hofer_orderable_leg}
Let $\tilde{\Lambda}_t$ be the lift of the isotopy $(\phi^{\alpha}_{s}(\Lambda))_{s\in[0,t]}$ to $\tilde{\mathcal{L}}$ for some contact form $\alpha$ and fixed $t\in\R$.
If $$d_{\alpha}(\tilde{\Lambda},\tilde{\Lambda}_t)<|t|,$$
there exists a positive contractible loop in $\mathcal{L}$.
\end{theorem}

\subsection{Liouville and Weinstein manifolds}

An important class of contact manifolds are those which admit a (strong) symplectic filling by a Weinstein, or more generally by a Liouville manifold.
In our conventions we mostly follow \cite{Cieliebak12, Eliashberg06}.

Let $(W,\lambda)$ be an exact symplectic manifold, i.e., $\lambda$ is a $1$-form such that $\omega:=d\lambda$ is a symplectic form.
Since $\omega$ is non-degenerate, there exists a unique vector field $Z$, called the \textbf{Liouville vector field}, such that $\omega(Z,\cdot)=\lambda$.

We call an exact symplectic manifold $(W,\lambda)$ a  \textbf{Liouville manifold} if
\begin{itemize}
    \item[(i)] The flow $f_t^Z$ of the Liouville vector field $Z$ is complete,
    \item[(ii)] There exists an open subset $U$ with compact closure and smooth boundary such that $W=U\sqcup \bigcup_{t\geq 0}f_t^Z(\partial U)$.
\end{itemize}

One can define the \textbf{skeleton} of the Liouville manifold as
$$\mathrm{Skel}(W,\lambda):=\bigcap\limits_{t<0} f_t^Z\left(\overline{U}\right).$$
Often we will just denote the skeleton of a Liouville manifold by the letter $L$.
Note that the skeleton does not depend on the choice of the open subset $U$.

The Liouville form $\lambda$ naturally descends to a contact form $\alpha$ on $M\cong\partial U$.
Even though $\alpha$ depends on the choice of $U$, the contact structure $\xi=\ker\alpha$ is independent of that choice.
In other words, the Liouville vector field induces a free $\R$-action on $W\setminus\mathrm{Skel}(W,\lambda)$ whose quotient inherits a contact structure from $\ker\lambda$ that is contactomorphic to $(M,\xi)$.
In particular $W\setminus\mathrm{Skel}(W,\lambda)$ is symplectomorphic to the symplectization of $(M,\xi)$.
For this reason we refer to $(M,\xi)$ as the \textbf{ideal contact boundary} of $(W,\lambda)$.

A Liouville manifold $(W,\lambda)$ is called a \textbf{Weinstein manifold} if there exists an exhausting Morse function $F\colon W\rightarrow\R$, i.e., $F$ is bounded from below and proper, which is gradient-like for $Z$.
By this we mean that for some $\delta>0$ and an auxiliary Riemannian metric 
$$dF(Z)\geq \delta\left(\vert Z\vert^2+\vert dF\vert^2\right).$$
While a general $2n$-dimensional Liouville manifold can have for example a disconnected boundary or a skeleton of dimension greater than $n$ \cite{McDuff91}, the boundary of a Weinstein manifold is connected and its skeleton a union of isotropic submanifolds of dimension at most $n$, which are the stable submanifolds of the critical points of the Liouville vector field, see e.g. \cite{Cieliebak12, Starkston18}.

We always assume that a Weinstein manifold $(W,\lambda)$ is of \textbf{finite type}, i.e., that the Morse function $F$ has finitely many critical points.
This ensures that the skeleton is the union of finitely many isotropic submanifolds.
In particular $L$ is an \textbf{isotropic complex}, i.e., it admits a filtration
$$L^{0}\subset \cdots\subset L^{n}=L,$$
where $L^{i}\setminus L^{i-1}$ is an isotropic submanifold of dimension $i$.

We denote by $L^{\mathrm{crit}}=L\setminus L^{n-1}$ the critical part of the skeleton.

\subsection{Open book decompositions and contact structures}\label{subsecOB}

Open book decompositions are an important tool in contact topology that allows to decompose a given contact manifold into a codimension $2$ contact submanifold and a family of codimension $1$ Liouville domains.
In the following we review the most important facts open book decompositions and contact structures.
For further details see e.g. \cite{Giroux00, Giroux02, Giroux20, Courte18, Geiges}.

An \textbf{open book decomposition} $(K,\theta)$ of a closed manifold $M$ consists of a codimension $2$ submanifold $K$, called the binding, and a smooth, locally trivial fibration $\theta\colon M\setminus K\rightarrow S^1$.
Additionally it is required that $K$ has a trivial tubular neighbourhood diffeomorphic to $K\times D^2$ such that in that neighbourhood $\theta$ corresponds to the angular coordinate on the disc $D^2$.
The fibres $\theta^{-1}(x)$ are called pages of the open book decomposition.
Note that if $M$ is oriented, an orientation for the pages can be chosen by requiring that the orientation of the page together with the orientation induced by $\theta$ via the standard orientation of $S^1$ induces the given orientation on $M$.
Since $K=\partial \theta^{-1}(x)$ for all $x\in S^1$, this induces an orientation on $K$.

As first noticed in the work of Giroux \cite{Giroux00,Giroux02, Giroux20}, the notion of an open book is closely related to contact structures.
Let $(M,\xi)$ be an oriented contact manifold with open book decomposition $(K,\theta)$.
Given a contact form $\alpha$ inducing the orientation of $M$, we say that $(M,\xi=\ker\alpha)$ is \textbf{supported} by the open book decomposition $(K,\theta)$ if
\begin{itemize}
    \item $d\alpha$ defines a symplectic form on each page inducing the given orientation.
    \item $\alpha$ defines a contact form on $K$ inducing the given orientation.
\end{itemize}

Note that the pages inherit the structure of a Liouville domain from the contact form $\alpha$.

We call two open books $(K,\theta), (\widehat{K},\widehat{\theta})$ supporting $(M,\xi)$ \textbf{equivalent} if there exists a contactomorphism $\phi \in \mathrm{Cont}(M,\xi)$ preserving the co-orientation of $\xi$ such that $\widehat{K} = \phi(K)$ and $\widehat{\theta} = \theta \circ \phi^{-1}.$

According to a very important result of Giroux \cite{Giroux02} (see also \cite{Colin08, Ibort00}; a new proof has recently appeared in the literature \cite{Honda24} based on \cite{Honda19, Eliashberg22}), any contact structure in dimension $3$ or higher is supported by an open book decomposition whose pages are Weinstein domains (see \cite{Cieliebak12} for details about this notion).

\begin{theorem}[\cite{Giroux02, Honda24}]
Let $(M,\xi)$ be a contact manifold of dimension at least $3$.
Then $(M,\xi)$ is supported by an open book decomposition with Weinstein pages.
\end{theorem}

The notion of ideal Giroux form \cite{Giroux20, Courte18} allows to consider the pages of a supporting open book in a more uniform setting.

An \textbf{ideal Giroux form} is a contact form $\beta$ on $M\setminus K$ with the following properties:
\begin{itemize}
    \item $d\theta(R_{\beta})>0$
    \item There exists a contact form $\eta$ on $K$ such that on a tubular neighbourhood of $K$ diffeomorphic to $K\times D^2$ one has 
    $$\beta=d\theta +\frac{\eta}{r^2}.$$
    Here $\eta$ is a contact form on $K$, $r$ denotes the radial coordinate on $D^2$ and the neighbourhood is chosen such that the fibration $\theta$ coincides with the angular coordinate on $D^2$.
\end{itemize}

The ideal Giroux form $\beta$ induces on each page the structure of a Liouville manifold as in \cite[Section 1.5]{Eliashberg06}.
Moreover its Reeb flow induces exact symplectomorphisms between the pages.
Up to a small deformation near the binding, any open book decomposition admits an ideal Giroux form, see \cite[Lemma 3.9]{Courte18}.

Note that due to its behaviour near the binding, an ideal Giroux form $\beta$ does not extend to a contact form on all of $M$.
However, the contact structure induced by $\beta$ coincides with the contact structure supported by $(K,\theta)$ and therefore extends over the binding.

\begin{lemma}[\cite{Courte18}]\label{lem_ob_pages}
Let $(K,\theta)$ be an open book supporting some contact structure $(M,\xi)$ and $\beta$ be an ideal Giroux form.
Let $(W,\lambda)$ be exact symplectomorphic to the pages.
For any $c\in S^1$ and any open interval $I$ containing $c$ there exists $\epsilon>0$ and an embedding $i\colon W\times (-\epsilon,\epsilon)\rightarrow \theta^{-1}(I)$ such that $i^{\ast}\beta=dt+\lambda$ and $i^{\ast}\theta=t+c$ ($\mod 2\pi$).
\end{lemma}

The following key Lemma is inspired by \cite[Proposition 3.5]{Courte18}.
For completeness, we provide a proof here.

\begin{lemma}\label{lem_ob_contraction}
Let $(K,\theta)$ be an open book decomposition supporting a contact structure $(M,\xi)$ and let $\beta$ be an ideal Giroux form.
Let $W_{\pm}$ be two different pages of the open book and $L_{\pm}$ be the skeletons of $W_{\pm}$ with respect to the Liouville forms induced by $\beta$.
Then there exists a contact vector field $X$ whose flow $\phi_t$ has the following property:

For any two open neighbourhoods $U_{\pm}$ of $L_{\pm}$ there exists a $T>0$ such that
$$\phi_T(M\setminus U_-)\subset U_+.$$
\end{lemma}

\begin{proof}
Without loss of generality assume that $W_-=\theta^{-1}(0)$ and $W_+=\theta^{-1}(\pi)$.

Let $V\cong K\times D^2$ be an open neighbourhood of $K$ such that on $V$ $\beta$ is of the form 
$$\beta=d\theta +\frac{\eta}{r^2}.$$
Let $f(r)$ be some non-decreasing function interpolating between $0<c<1$ near $0$ and $r$ near $1$.
Consider the function $H=f(r)r\sin\theta=f(r)y$.
We define $X$ as the contact vector field defined by the contact Hamiltonian $\sin\theta$ with respect to the contact form $\beta$ on $M\setminus V$ and defined by the contact Hamiltonian $H$ with respect to the contact form $r^2\beta$ on $V$.
The vector field $X$ is smooth since near $\partial V$ one has that $r^2\sin\theta=H$.

Let $(W,\lambda)$ be exact symplectomorphic to the pages.
Lemma \ref{lem_ob_pages} allows to choose $\epsilon>0$ and embeddings with disjoint images $i_{\pm}\colon W\times(-\epsilon,\epsilon)\rightarrow M$ as described above such that $i_{\pm}(W\times\{0\})=W_{\pm}$.

On $i_{\pm}(W\times(-\epsilon,\epsilon)) $ the vector field $X$ is given by $X=\cos \theta (i_{\pm})_{\ast}Z+\sin \theta R_{\beta}$.
Here $Z$ denotes the Liouville vector field on $W$.

An easy computation shows that on $V$ we have that
$$X=\frac{rf-r^2f'}{2}\sin\theta R_{\eta}+\frac{rf'+f}{2r}\sin\theta\partial_{\theta}-\frac{f}{2}\cos\theta\partial_r.$$
It follows that on $V$
$$dx(X)=d(r\cos\theta)(X)=-\frac{rf'\sin\theta+f}{2}\leq -\frac{c}{2},$$
and on $V\setminus K$
$$d\theta(X)=\frac{rf'+f}{2r}\sin\theta\geq \frac{c}{2}\sin\theta.$$

Note that the zeros of $X$ are precisely $L_\pm$:
By definition the vector field $X$ is nowhere vanishing on $M\setminus (V\cup W_{\pm})$ and by the considerations above it is nowhere vanishing on $V$.
On $W_{\pm}$ the vector field $X$ is proportional to the Liouville flow.

As on $V$ one has $dx(X)<0$ and
$$dy(X)=d(r\sin\theta)(X)=\frac{rf'}{2}\cos\theta\sin\theta,$$
one can see that a flow line $\phi_t(p)$ intersects $K$ if and only if $p\in W_-\setminus L_-$, and that in this case $\phi_t(p)$ enters $W_+$ after crossing $K$.

Now pick a domain $P\subset W$ with smooth star-shaped boundary and $\epsilon$ small enough such that $A_+:=i(P\times (-\epsilon,\epsilon))\subset U_+$.

{\bf Claim:} For any $p\in M\setminus L_-$ there exists a $T>0$ such that $\phi_t(p)\in A_+$ for all $t\geq T$.
Then the Lemma follows from the compactness of $M\setminus U_-$.

1. case: $p\in (W_-\setminus L_-)\cup K$.
In this case $\phi_t(p)$ is contained in $W_-\cup W_+\cup K$ and given by the Liouville flow on $W_-\setminus V$ and the inverse of the Liouville flow on $W_+\setminus V$.
Since $\partial (V\cap W_{\pm})$ is a star-shaped hyper surface which is the boundary of a compact domain in $W_{\pm}$, the Liouville flow reaches it in finite time.
As shown above $\phi_t(p)$ exits $V$ after a finite time.

2. case: $p\notin (W_-\setminus L_-)\cup K$.
Assume that $p\notin W_+$.
On $M\setminus V$ one can compute
$$d\theta(X)=\sin\theta d\theta(R_{\beta}).$$
Together with the computations on $V$ this implies that along $\phi_t(p)$, $d\theta$ is bounded from either below or above on $M\setminus i_+(W\times (-\epsilon,\epsilon))$ depending on the sign of $\sin\theta(p)$.
In particular $\phi_t(p)$ reaches $i_+(W\times (-\epsilon,\epsilon))$ in finite time and leaves $V$ in finite time.
As on $i_+(W\times (-\epsilon,\epsilon))$ the projection of the flow of $X$ is proportional to the inverse Liouville flow (with a factor close to $1$), similar to the first case the flow line $\phi_t(p)$ reaches $A_+$ in finite time.

\end{proof}

\subsection{Loose Legendrians}

A class of Legendrians with particularly flexible properties, called loose Legendrians, was introduced by Murphy in \cite{Murphy12}, see for example also \cite{Cieliebak12, Casals19, Courte18} for further details on loose Legendrians.

Consider the standard contact $3$-ball $B^3(2)$ of radius two with contact form $\alpha_0=dz-xdy$ and the front projection $(x,y,z)\mapsto (y,z)$.
Let $\gamma$ be the Legendrian arc whose front projection is a zig-zag as shown in Figure \ref{figzigzag}.
For $r>1$ consider the sets $V_r:=\{(q,p)\in T^{\ast}R^n||q|<r,|p|<r\}$ and $Z_r:=V_r\cap\{p=0\}$.
Then $L_r:=\gamma\times Z_r$ is a Legendrian in $ B^3(2)\times V_r$ equipped with the contact form $\alpha_0+\lambda$, where $\lambda$ denotes the standard Liouville form.

A connected Legendrian $L$ in a contact manifold $(M^{2n+1},\xi)$, where $n\geq 2$, is called \textbf{loose} if there exist an open subset $U\subset M$, $r>1$ and a contactomorphism $\phi\colon U\rightarrow  B^3(2)\times V_r$ with $\phi(U\cap L)=L_r$.
We refer to the pair $(U,U\cap L)$ as a \textbf{loose chart} for $L$.

We define a disconnected Legendrian to be loose if every connected component is loose in the complement of the others.

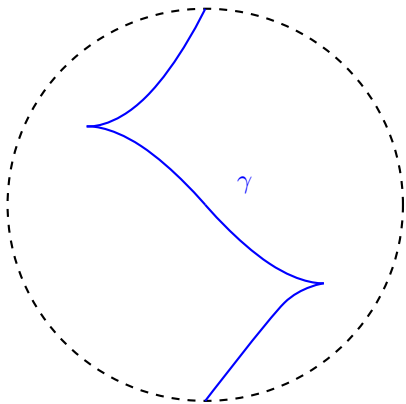
\begin{figure}
    \centering
  \begin{tikzpicture}[scale=1.3]
    
     \draw[dashed, thick] (0,0) ellipse (2cm and 2cm);
   
     \draw[thick,blue ] plot[smooth, tension=0.5] coordinates {
      (0,-2)
      (0.8,-1)
      (1.2,-0.8)
    };

     \draw[thick,blue ] plot[smooth, tension=1] coordinates {
      (1.2,-0.8)
      (0.65, -0.6)
      (0,0)
    };

         \draw[thick,blue ] plot[smooth, tension=1] coordinates {
      (0,0)
      (-0.65,0.6)
      (-1.2,0.8)
    };

         \draw[thick,blue ] plot[smooth, tension=1] coordinates {
      (-1.2,0.8)
      (-0.6,1.1)
      (0,2)
    };

    \node[thick,blue] at (0.4,0.2) {$\gamma$};
\end{tikzpicture}
    \caption{Front projection of the Legendrian arc $\gamma$.}
    \label{figzigzag}
\end{figure}

As shown in \cite{Murphy12}, loose Legendrians satisfy an $h$-principle, which can be used to prove various flexibility results.
In this paper we will use the following application of the $h$-principle to the contact Hofer metric on the isotopy class of a loose Legendrian which the simplification of a theorem proved by Nakamura \cite[Theorem 1.2]{Nakamura21}.

\begin{theorem}[\cite{Nakamura21}]\label{thm:naka}
 Let $(M,\ker\alpha)$ be a closed contact manifold and $\Lambda$ a closed loose Legendrian in $M$.
 Let $(\phi_t)_{t\in [0,1]}$ be an isotopy of contactomorphisms starting at the identity.
 Suppose that there exists a loose chart $U$ for $\Lambda$ such that the contact Hofer norm is bounded by $c_1>0$ on $\mathcal{G}(U)$ and by $c_2>0$ on $\mathcal{G}(\phi_1(U))$.
 Then for any $\delta>0$ there exists a contact isotopy $(\psi_t)_{t\in [0,1]}$ of length $l_{\alpha}(\psi_t)\leq c_1+c_2+\delta$ such that $\psi_1(\Lambda)=\phi_1(\Lambda)$ and $\psi_t(\Lambda)$ is $\mathcal{C}^0$-close to $\phi_t(\Lambda)$ for each $t$.
\end{theorem}

\begin{remark}
 It follows from Proposition \ref{prophandle} that such an upper bound $c_1$ exists for any loose chart, see also \cite[Remark 1.3]{Nakamura21}.  
\end{remark}

We will use the following notion of a "singular Legendrian submanifold". A subset $L$ of a contact manifold $M$ of dimension $2n+1$ is called an {\bf isotropic complex} if it admits a filtration \[\emptyset = L^{-1} \subset L^0 \subset \ldots \subset L^{n-1} \subset L^n = L\] such that for all $0 \leq i \leq n,$ $L^i \setminus L^{i-1}$ is an isotropic submanifold of $M$ without
boundary of dimension $i$ and each point of $L^n \setminus L^{n-1}$ has an open neighbourhood $O$ such that $O \cap L = O \cap (L^n \setminus L^{n-1}).$ 

Note that given a contact manifold that is supported by an open book decomposition with pages exact symplectomorphic to a finite type Weinstein manifold, the skeletons of each page forms an isotropic complex in $M$.

An isotropic complex is called loose if $L \setminus L^{n-1}$ is loose in $M \setminus L^{n-1}.$ We remark that Theorem \ref{thm:naka} readily extends to this setting as follows: 

\begin{theorem}\label{thm:naka2}
 Let $(M,\ker\alpha)$ be a closed contact manifold and $L$ a loose isotropic complex in $M$.
 Let $(\phi_t)_{t\in [0,1]}$ be an isotopy of contactomorphisms starting at the identity and $\phi_t = \id$ for all $t$ on a neighbourhood $U$ of $L^{n-1}.$ Suppose that there exists a collection of loose charts $V$ for $L \setminus U$ in $M \setminus U$ such that the contact Hofer norm is bounded by $c_1>0$ on $\mathcal{G}(U)$ and by $c_2>0$ on $\mathcal{G}(\phi_1(U))$. Then for any $\delta>0$ there exists a contact isotopy $(\psi_t)_{t\in [0,1]}$ of length $l_{\alpha}(\psi_t)\leq c_1+c_2+\delta$ such that $\psi_t = \id$ for all $t$ on $U,$ $\psi_1(L)=\phi_1(L),$ and $\psi_t(L)$ is $\mathcal{C}^0$-close to $\phi_t(L)$.
\end{theorem}

\section{Bounds on the contact Hofer norm}

\subsection{The contact Hofer diameter of skeleton-complements}

In this section we bound the contact Hofer norm for contactomorphisms compactly supported in the complement of the skeleton of a page of an open book decomposition supporting a contact structure.
Let $(K,\theta)$ be an open book decomposition supporting a closed contact manifold $(M,\xi)$, see \ref{subsecOB}.

Lemma \ref{lem_ob_contraction} shows that given two different pages $W_{\pm}$ with skeletons $L_{\pm}$, there exists a vector field $X$ on $M$ whose flow $\psi_t$ satisfies that for every open neighbourhood $U$ of $L_-$
\begin{align*}
\bigcap\limits_{t\geq 0}\psi_t(M\setminus U)=L_+.
\end{align*}

\begin{prop}\label{prophandle}
Let $V=M\setminus L_-$, then the contact Hofer norm (with respect to any contact form that extends to $M$) is bounded on $\Tilde{\mathcal{G}}(V)$ and hence on $\mathcal{G}(V)$.
\end{prop}

\begin{proof}
Note that $W_{\pm}$ has a neighbourhood that is contactomorphic to $(-\epsilon,\epsilon)\times W$ equipped with the contact form $dt+\lambda$, where $(W,\lambda)$ is exact symplectomorphic to the pages of the open book.

\textbf{Claim:} Let $(U,\alpha)=((-\epsilon,\epsilon)\times W, dt+\lambda)$.
Then $|\cdot|_{\alpha}$ is bounded on $\tilde{\mathcal{G}}(U)$.

Let $(f_s)_{s\in[0,1]}$ be a contact isotopy in $U$ with support in some compact subset $K_0\subset (-\epsilon_0,\epsilon_0)\times W$ and let $K_1\subset (-\epsilon_1,\epsilon_1)\times W$ be a compact subset such that $K_0$ is contained in the interior of $K_1$.
Here we assume that $0<\epsilon_0<\epsilon_1<\epsilon$.
Consider the contact vector field $X$ with contact Hamiltonian $\alpha(X)=-t$ and flow $(\phi_s)_{s\geq 0}$.
Hence $X=-t\partial_t-Z$, where $Z$ denotes the Liouville vector field of $\lambda$.
Then $\phi_s(t_0,p_0)=(t_0e^{-s},\eta_s(p_0))$, where $\eta_s$ denotes the flow of $Z$ at time $s$.
Using an appropriate cut-off function, we define a contact Hamiltonian $h_s$ such that for each $s\geq 0$ $h_s|_{\phi_s(K_0)}=\alpha(X)$,  $h_s=0$ outside of $\phi_s(K_1)$ and $\max\limits_U|h_s|\leq \epsilon_0e^{-s}+\delta e^{-s}$ for some $\delta>0$ arbitrarily small.
Let $(g_s)_{s\geq 0}$ be the isotopy generated by $h_s$.
Note that $g_s|_{K_0}=\phi_s$ for each $s$.
Moreover, 
$$|\tilde{g}_s|_{\alpha}\leq \epsilon_0+\delta$$
for all $s>0$.

One can easily compute that $\phi_s^{\ast}\alpha=e^{-s}\alpha$, i.e., on $K_0$ we have
$$g_s^{\ast}\alpha|_{K_0}=e^{-s}\alpha|_{K_0}.$$
It follows using the properties of the contact Hofer norm that
\begin{align*}
 |\tilde{f}_1|_{\alpha}&= |\tilde{g}_s\tilde{g}^{-1}_s\tilde{f}_1\tilde{g}_s\tilde{g}_s^{-1}|_{\alpha}=|\tilde{g}_s (g_s^{-1}\tilde{f}_1 g_s)\tilde{g}_s^{-1}|_{\alpha}\leq 2|\tilde{g}_s|_{\alpha}+|g_s^{-1}\tilde{f}_1 g_s|_{\alpha}\\
 &=2|\tilde{g}_s|_{\alpha}+|\tilde{f}_1|_{g_s^{\ast}\alpha}\leq 2\epsilon_0+2\delta+e^{-s}|\tilde{f}_1|_{\alpha}. 
\end{align*}

Here we use that for fixed $s\geq 0$ the path $(g_s^{-1}f_rg_s)_{r\in [0,1]}$ is homotopic with fixed endpoints to the path $(g_{rs}^{-1}f_rg_{rs})_{r\in[0,1]}$ and denote by $g_s^{-1}\tilde{f}_1 g_s= \tilde{g}^{-1}_s\tilde{f}_1\tilde{g}_s$ its class in the universal cover.

Choosing $s$ arbitrarily large and $\delta$ small, we obtain
$$|\tilde{f}_1|_{\alpha}\leq 2\epsilon.$$

Let $\eta$ be the restriction of a contact form on $M$ to $V$.
Now let $(\psi_t)_{t\in\R}$ be the flow on $V$ contracting any neighbourhood of $L_-$ onto $L_+$.
Pick open neighbourhoods $U_{\pm}$ of $L_{\pm}$ such that $U_{\pm}$ is contactomorphic to $(-\epsilon,\epsilon)\times W$ equipped with the contact form $\alpha=dt+\lambda$.
Then $\eta|_{U_{\pm}}=\rho_{\pm}\alpha$, where $\rho_{\pm}$ is a bounded positive function on $U_{\pm}$. 
Let $\Tilde{U}\subset U_-$ be an open neighbourhood of $L_-$ whose closure is contained in $U_-$.
Since $V\setminus \Tilde{U}$ has compact closure, there exists a $T$ such that $\psi_T(V\setminus \Tilde{U})\subset U_+$.

Let $(f_s)_{s\in[0,1]}$ be a contact isotopy in $V$ that is supported in some compact subset $K$.
Up to modifying the contact Hamiltonian $h(t,p)=t$ on $U_-$, we can use its flow to move the support of $f_s$ into $V\setminus \Tilde{U}$ for each $s$.
Hence, using the claim and the fact that $\rho_{\pm}$ is bounded, we get
$$|\tilde{f}_1|_{\eta}\leq |\tilde{F}_1|_{\eta}+C_-,$$
where $(F_s)_{s\in[0,1]}$ is an isotopy with support in $V\setminus \Tilde{U}$ and $C_-$ is a constant depending on the choice of the neighbourhood $U_-$.
Moreover, the isotopy $(\psi_TF_s\psi_T^{-1})_{s\in[0,1]}$ is compactly supported in $U_+$, i.e., using the claim again we obtain
$$|\tilde{f}_1|_{\eta}\leq 2|\tilde{\psi}_T|_{\eta}+C_++C_-,$$
where $C_+$ is a constant depending on the choice of the neighbourhood $U_+$.

\end{proof} 

\subsection{A Fragmentation result}\label{secfrag}

Fragmentation results allow us to decompose a given contact isotopy into contact isotopies compactly supported in complements of skeletons.
The following elementary fragmentation result was proved in \cite{Fraser} for the case of points in the standard contact sphere and in \cite{Courte18} for general contact manifolds.

\begin{lemma}\label{lemfragmentation}
Let $K_{\pm}\subset M$ be compact subsets of a contact manifold $(M,\xi)$ (open or closed).
Let $(f_t)_{t\in[0,1]}$ be a compactly supported contact isotopy such that $f_t(K_-)\cap K_+=\emptyset$ for all $t\in [0,1]$.
Then $f_t=f_t^-\circ f_t^+$, where $f_t^{\pm}$ is compactly supported in $M\setminus K_{\pm}$.
\end{lemma}

\begin{proof}
Since $K_{\pm}$ are compact, we can find an open neighbourhood $U$ of $K_-$ such that the set
$$A:=\bigcup\limits_{t\in [0,1]}f_t(U)$$
does not intersect $K_+$.
Choose an open neighbourhood $V$ with compact closure such that $\overline{A}\subset V$ and $V\cap K_+=\emptyset$.
By cutting off the contact Hamiltonian generating the path $f_t$, we can find a path $f^+_t$ such that $f^+_t=f_t$ on $f_t(U)$ and $f^+_t=\mathrm{id}$ on $M\setminus V$.
By definition $f^+_t$ is compactly supported in $M\setminus K_+$.
Define $f_t^-:=f_t\circ (f_t^+)^{-1}$.
Then $f_t^-$ is compactly supported in $M\setminus K_-$ and $f_t=f_t^-\circ f_t^+$.
\end{proof}

\begin{remark}
Lemma \ref{lemfragmentation} in general holds only for contact isotopies.
Given a contactomorphism $f$ such that $f(K^-)\cap K^+=\emptyset$ it is not clear if there exists a path $f_t$ connecting $f$ to $\mathrm{id}$ such that $f_t(K^-)$ does not intersect $K^+$.
\end{remark}

\subsection{Proof of Theorem \ref{thm1} (i)}\label{sec:pf thm1i}
Let $\beta$ be an ideal Giroux form for the open book $(K,\theta)$ and $W_{\pm}:=\theta^{-1}(\pm\pi)$ pages with skeletons $L_{\pm}$.
Since $W$ is a subcritical Weinstein manifold, $L_{\pm}$ is the union of isotropic submanifolds of dimension smaller than $n$.
Given an isotopy $(f_t)_{t\in [0,1]}$ starting at $\mathrm{id}$, we can thus find an isotopy $(f_t')_{t\in [0,1]}$ which is $\mathcal{C}^0$-close to $f_t$ such that $f_t'(L_-)\cap L_+=\emptyset$ for all $t\in [0,1]$, see e.g. \cite[Proposition 5.2]{Courte18}.

Proposition \ref{prophandle} and Lemma \ref{lemfragmentation} imply that there exists a $C>0$ such that $\vert \tilde{f'}_1\vert_{\alpha}\leq C$ in $\Tilde{\mathcal{G}}$.
As we can choose $f_t'$ arbitrarily close to $f_t$, the same bound holds for $(f_t')^{-1}f_t$.

It follows that 
$$\vert \tilde{f}_1\vert_{\alpha}=\vert \tilde{f'}_1\tilde{f'}_1^{-1}\tilde{f}_1\vert_{\alpha}\leq \vert \tilde{f'}_1\vert_{\alpha}+ \vert \tilde{f'}_1^{-1}\tilde{f}_1\vert_{\alpha}\leq 2C. $$
The non-orderability of $M$ follows from Theorem \ref{thm_hofer_orderable}.

\subsection{Proof of Theorem \ref{thm1} (ii)}

Suppose that the skeleton $L_-$ of some page $W_-$ is a loose Legendrian in $M$ and let $U$ be a loose chart for $L_-$.
As $U$ is a Darboux ball contactomorphic to a ball in the standard $\R^{2n+1}$, Proposition \ref{prophandle} implies the existence of a constant $C_1$ such that the contact Hofer norm on $\tilde{\mathcal{G}}(U)$ is bounded by $C_1$.

Let $\phi_t$ be a path of contactomorphisms such that $\phi_1$ is strict.
Then $\phi_1(L_-)$ is a loose Legendrian with loose chart $\phi_1(U)$.
Since $\phi_1$ is strict, the contact Hofer norm on $\tilde{\mathcal{G}}(\phi_1(U))$ is bounded by $C_1$.

Since $L$ is loose, \cite[Theorem 1.2]{Nakamura21} implies that for any $\delta>0$ there exists a contact isotopy $(\psi_t)_{t\in[0,1]}$ of contact Hofer length smaller than $\delta+2C_1$ such that $\psi_1(L_-)=\phi_1(L_-)$ and such that $\psi_t(L_-)$ is $\mathcal{C}^0$-close to $\phi_t(L_-)$.

It follows that $(\psi_t^{-1}\phi_t)(L_-)$ is $\mathcal{C}^0$-close to $L_-$, in particular there exists a page $W_+$ (independent of $\phi_t$) with skeleton $L_+$ such that $(\psi_t^{-1}\phi_t)(L_-)\cap L_+=\emptyset$ for all $t$.
Proposition \ref{prophandle} and Lemma \ref{lemfragmentation} imply that $\vert\tilde{\psi}_1^{-1}\tilde{\phi}_1\vert_{\alpha}\leq C_2$ for some constant $C_2$ independent of $(\phi_t)_{t\in [0,1]}$.
Hence
$$\vert\tilde{\phi}_1\vert_{\alpha}=\vert\tilde{\psi}_1\tilde{\psi}_1^{-1}\tilde{\phi}_1\vert_{\alpha}\leq \vert\tilde{\psi}_1\vert_{\alpha}+\vert\tilde{\psi}_1^{-1}\tilde{\phi}_1\vert_{\alpha}\leq \delta+2C_1+C_2.$$

This proves Theorem \ref{thm1} (ii).
The non-orderability of $M$ follows from Theorem \ref{thm_hofer_orderable} and the fact that the Reeb flow is strict.

\subsection{Proof of Corollary \ref{cor1stab}}

First note that the ideal contact boundary of $W\times\C$ admits an open book decomposition whose page is exact symplectomorphic to $W$ and the monodromy is isotopic to the identity map.

Let $L$ be the skeleton of $W$.
It remains to prove the case when $W$ is critical, i.e., when $L$ is a Lagrangian.

It follows from \cite[Theorem 1.1]{Starkston18} and \cite[Corollary 11.21]{Cieliebak12} that $W$ is exact symplectomorphic to $T^{\ast}L$.
Then for $p\in L$ the fibre over $p$ defines a Lagrangian $\R^n$ that is Legendrian at infinity and intersects $L$ in a unique point.
Hence \cite[Proposition 2.9]{Casals19} implies that the skeleton of the page of the open book, which is given by a lift of $L$, is loose in the ideal contact boundary of $W\times\C$.

The result follows from Theorem \ref{thm1}.

\subsection{Proof of Theorem \ref{thm2}}

Let $L\subset M$ be the skeleton of the page of some open book.
Suppose there exists a path of contactomorphisms $(\phi_t)_{t\geq 0}$ such that $L$ can be displaced from $$A:=\bigcup\limits_{t\geq 0}\phi_t(L),$$
i.e., there exists a contactomorphism $\psi$ with $\psi(L)\cap A=\emptyset$.
By Lemma \ref{lemfragmentation} we have $\phi_t=g_t\circ h_t$, where $g_t$ is compactly supported in $M\setminus L$ and $h_t$ in $M\setminus \psi(L)$.

Note that $\psi(L)$ is the skeleton of some other open book decomposition.
In particular, by Proposition \ref{prophandle}, the contact Hofer norms $\vert \tilde{g}_t\vert_{\alpha}$ and $\vert \tilde{h}_t\vert_{\alpha}$ can be bounded from above independent of $t$.
Thus, $\vert \tilde{\phi}_t\vert_{\alpha}$ is bounded from above independent of $t$.

If $\phi_t$ is a Reeb flow, Theorem \ref{thm_hofer_orderable} implies the existence of a positive contractible loop.

\subsection{Proof of Corollary \ref{corst}}

In this section we show that $T^n\times S^{n+1}$ with its canonical contact structure is non-orderable.

Consider $T^{\ast}T^n$ and $\C$ with their canonical Liouville forms $\lambda_{T^n}$ and $\lambda_{\C}=\frac{1}{2}(xdy-ydx)$.
Then the $1$-stabilization of $T^{\ast}T^n$ is given by $T^{\ast}T^n\times\C$ equipped with the Liouville form $\lambda_{T^n}+\lambda_{\C}$.
The Liouville form on $T^{\ast}T^n\times\C$ induces a contact form on its ideal boundary, which is diffeomorphic to $T^n\times S^{n+1}$.
We refer to the resulting contact structure as the standard contact structure on $T^n\times S^{n+1}$.

Consider $T^n\times S^{n+1}$ as a subset of $T^n\times \R^n\times \C$, where $S^{n+1}$ is considered as the unit sphere in $\R^n\times\C\cong \R^{n+2}$.
A contact form for the standard contact structure on $T^n\times S^{n+1} $ is given by the restriction of the $1$-form
$$\lambda=\lambda_{T^n}+\lambda_{\C}=\sum\limits_kx_kd\theta_k+\frac{1}{2}(xdy-ydx).$$
We denote this restriction by $\alpha$.
Here we consider $T^n=(S^1)^n$ with coordinates $(e^{2\pi i\theta_1},\cdots , e^{2\pi i\theta_n})$, $\R^n$ with coordinates $(x_1,\cdots,x_n)$ and $\C$ with coordinate $z=x+iy$.

An explicit open book with page $T^{\ast}T^n$ and monodromy being the identity map can be constructed as follows:
Let 
$$B:=\{(p,x_1,\cdots,x_n,z)\in T^n\times \R^n\times \C|z=0,x_1^2+\cdots+x_n^2=1\}\subset T^n\times S^{n+1}.$$
Then the map
\begin{align*}
    P\colon T^n\times \R^n\times \C^{\ast}&\rightarrow S^1,\\
    (p,x_1,\cdots,x_n,z)&\mapsto \arg(z)
\end{align*}
restricts to a fibration of $T^n\times S^{n+1}\setminus B$ over $S^1$.
The submanifold $B$ admits a neighbourhood diffeomorphic to $D^2\times B$ such that $B=\{0\}\times B$ and $P$ is given by the angular coordinate in $D^2\setminus \{0\}$, i.e., $(B,P)$ defines an open book decomposition of $T^n\times S^{n+1}$.
As in the above coordinates
$$\lambda|_B=\sum\limits_kx_kd\theta_k,$$
the contact form $\alpha$ induces a contact form on $B$ such that $(B,\ker \alpha|_B)$ is contactomorphic to the standard spherical cotangent bundle of $T^n$.

Let $\Sigma=P^{-1}(1)$ be a page of the open book.
Then 
$$\Sigma=\{(p,x_1,\cdots,x_n, z)\in T^n\times\R^n\times (0,1]|x_1^2+\cdots+x_n^2+x^2=1\}.$$
A simple computation in the above coordinates verifies that $d\lambda|_{\Sigma}$ is a symplectic form on $\Sigma$ defining its positive orientation.
Since $d\lambda$ is invariant under rotations in the $\C$-factor this holds for every page of the open book.
Hence $(B,P)$ defines an open book adapted to the standard contact structure on $T^n\times S^{n+1}$.
Note that the map 
\begin{align*}
    f\colon \Sigma&\rightarrow T^n\times B^n(1),\\
    (p,x_1,\cdots,x_n,z)&\mapsto (p, x_1,\cdots, x_n)
\end{align*}
defines an exact symplectomorphism between $\Sigma$ and the unit disc cotangent bundle of $T^n$.
The $1$-form $\lambda$ induces a Liouville form on $\Sigma$ with Liouville vector field
$$Y=\sum\limits_kx_k\partial_{x_k}+\frac{x^2-1}{x}\partial_x.$$
Thus the skeleton of $\Sigma$ is given by the torus 
$$L:=T^n\times \{0\}\times \{1\}\subset T^n\times \R^n\times\C.$$
Note that the image of this torus under the Reeb flow of $\alpha$ is 
$$A:=\bigcup\limits_{t\in\R}\phi_t^{\alpha}(L)=\{(p,x_1,\cdots,x_n,z)\in T^n\times S^{n+1}||z|^2=1\}.$$
The torus $A$ is a pre-Lagrangian toric fibre of the standard contact action of $T^{n+1}$ on $T^n\times S^{n+1}$, see \cite[Section 4.2]{Marincovic16} and can be displaced from itself by a contactomorphism due to \cite[Theorem 1.2]{Marincovic16}.
Hence Theorem \ref{thm2} implies the existence of a positive contractible loop.

\subsection{Proof of Corollary \ref{corekp}}\label{sec:ekp}

Let $W$ be a Liouville manifold and $(M,\xi)$ be the ideal boundary of the $2$-stabilization $W\times \C^2$.
Clearly $(M,\xi)$ has a supporting open book decomposition $(K,\theta)$ and an ideal Giroux form, such that the pages are exact symplectomorphic to $W\times\C$, and whose monodromy is isotopic to the identity, see e.g. \cite{Giroux20} or the Proof of Corollary \ref{corst}.

Hence we can choose a contact form $\alpha$ whose Reeb flow is periodic near the skeletons of each page; in particular the image of a skeleton under the Reeb flow is given by the union of all skeletons.
Due to Lemma \ref{lem_ob_pages}, it suffices to displace the skeleton of $W\times \C$ in $(-\epsilon,\epsilon)\times W\times\C$ equipped with the contact form $\alpha=dt+\lambda$, where $\lambda$ denotes the Liouville form on $ W\times\C$.

Note that the skeleton of a $1$-stabilization $W\times\C$ is given by $L\times\{0\}$, where $L$ is the skeleton of $W$.
In particular, the image of a skeleton under the Reeb flow is $(-\epsilon,\epsilon)\times L\times\{0\}$.

The skeleton can be displaced from this image by the lift of a translation in the $\C$ factor to $(-\epsilon,\epsilon)\times W\times\C$:
Let $(x,y)$ be coordinates on $\C$ and consider the contact Hamiltonian $h(t,p,x,y)=2y$.
Then $h$ induces the contact vector field $-2\partial_x+y\partial_t$ whose flow displaces the skeleton  $\{0\}\times L\times\{0\}$.
Cutting out $h$ outside of a ball around the skeleton yields a compactly supported contactomorphism displacing the skeleton from its image under the Reeb flow.

The result follows from Theorem \ref{thm2}.

\subsection{Proof of Theorem \ref{thm:perreeb}}

\begin{lemma}\label{lem_contract_loose}
Let $L$ be a loose Legendrian with loose chart $V$ around a point $p\in L\cap V$.
Furthermore, let $U\subset V$ be an open neighbourhood of $p$ and $W$ a Darboux ball with $\overline{V}\subset W$.
Then there exists a contactomorphism $\phi\in\mathcal{G}$ compactly supported in $W$ such that $\phi(V)\subset U$ and $\phi(L)\subset L\cup U$.
In particular $\phi(L)$ is a loose Legendrian with loose chart $\phi(V)$.
\end{lemma}

\begin{proof}
Without restriction, we can assume that $L\subset \R^{2n+1}$ equipped with the contact form $\alpha:= dz-ydx$.
Note that for simplicity, we omit the indices of the coordinates $(x_i,y_i)$, where $1\leq i\leq n$.
Assume further that $p=0$ and that $U, W$ are balls of radius $R_1< R_2$, respectively.
By the definition of a loose chart, we can further assume that outside of $V$ the Legendrian $L$ coincides with the set $\{z=y=0\}$.

Look at the contact Hamiltonian $h=f(x,y,z)(-2z+xy)$, where $f=1$ on $V$ and $f=0$ outside of $W$.
It generates a contact vector field of the form
$$X_h=(h+yg_1(x,y,z))\partial_z+g_1(x,y,z)\partial_x+g_2(x,y,z)\partial_y.$$
Since $d\alpha(X,\cdot)+dh=\partial_zh\alpha$, a simple computation shows that
$$g_2=\partial_xh+y\partial_zh=\partial_xf(-2z+xy)+fy+y\partial_zh.$$
In particular it follows that $g_2=0=h+yg_1(x,y,z)$ on $\{y=z=0\}$, i.e., the flow of $X_h$ preserves $L$ outside of $V$.
On $V$ the flow of $X$ coincides with the scaling map $t\mapsto (e^{-t}x,e^{-t}y,e^{-2t}z)$ and squeezes $V$ into $U$ for $t$ large enough.
\end{proof}

Write $f_t = \phi^{\alpha}_{t}$ for $t\geq 0$ and $\mathbb{L}'$ for the image of $L^{n-1}$ under the Reeb flow.
We start by suitably adjusting the open book to $f_t$ as follows.

\begin{lemma}\label{lem_ob_loose}
Let $(K,\theta)$ be the open book from Theorem \ref{thm:perreeb} and let $f_t$ be as above. Let $L$ be the skeleton of one of the pages of the open book (it is a loose isotropic complex).
There exists an equivalent open book decomposition $(\hat{K},\hat{\theta})$ with skeleton $\hat{L}$ of a page such that $\hat{L}^{n-1}=L^{n-1}$, and each connected component $\hat{L}_i^{\mathrm{crit}}$ admits a loose chart $V_i$ with $V_i\cap \hat{\mathbb{L}}'=\emptyset$. In particular, the image of $\hat{L}^{n-1}$ under the Reeb flow of $\alpha$ is still $\hat{\mathbb{L}}'=\mathbb{L}'$.
Moreover, one can assume that $V_i\cap V_j=\emptyset$ for $i\neq j$.
\end{lemma}

\begin{proof}
Since the Reeb flow is periodic near $L^{n-1}$, \cite[Proposition 3.11]{Courte18} implies that the intersection of $\mathbb{L}'$ with $L$ does not contain any sets that are open in $L$.
Therefore, for each connected component $L_{i}^{\mathrm{crit}}, i\in\{1,\cdots,k\}$ of the critical locus $L^{\mathrm{crit}}$ of $L,$ we can find a loose chart $U_i$ and a point $p_i\in U_i \cap L_{i}^{\mathrm{crit}}$ such that an open ball $U_i'$ around $p_i$ with closure contained in $U_i$ does not intersect $\mathbb{L}'$.
Moreover, we can assume that $U_i\cap U'_j=\emptyset$ for all $i\neq j.$ For each $i,$ Lemma \ref{lem_contract_loose} provides us with a contactomorphism $\phi_i\in \mathcal{G}$ of $M$ compactly supported in an arbitrarily small neighbourhood of $U_i$ inside the complement of $L^{n-1}$ such that $V_i:=\phi_i(U_i)\subset U_i'$ and $\phi_i(L_i^{\mathrm{crit}}) \cap L_j^{\mathrm{crit}} = \emptyset$ for all $j \neq i.$ Note that one can assume that $\phi_i$ is supported away from $L_{j}^{\mathrm{crit}}$ for $j\neq i$.
Consider $\phi:= \phi_1\cdots\phi_k$.
Then $\phi(L)$ is a loose isotropic complex with disjoint loose charts $V_i$ and $(\phi(L))^{n-1}=L^{n-1}$.
In particular, the charts $V_i$ are disjoint from  $\mathbb{L}'$.
Set $(\hat{K},\hat{\theta})=(\phi(K),\theta\circ \phi^{-1})$.
\end{proof}

By Lemma \ref{lem_ob_loose} we can assume that the open book has pages $W_{\pm}$ with skeletons $L_{\pm}$ such that $\mathbb{L}_-'$ does not intersect the loose charts $V_i$ of the connected components $L_{-,i}^{\mathrm{crit}}$.
Set $V = \bigcup_i V_i$ and define the size of $V$ to be $C(V) = \max_i C(V_i),$ where $C(V_i) = \sup |\widetilde{\theta}|_{\al}$ over all $\widetilde{\theta} \in \widetilde{\mathcal{G}}(V_i).$

Now, following the arguments from Sections \ref{sec:pf thm1i} and \ref{sec:ekp}, there exists a neighbourhood $U$ around $\mathbb{L}'$, disjoint from $V$, and a Hofer-small perturbation $f'_t = g_t f_t$ of $f_t$ with $|\tilde{g}_t|_{\al}<\delta$ and $\mathrm{supp}(g_t) \subset U$ such that $f'_t(L_-^{n-1}) \cap L_+ = \emptyset$ for all $t.$ Hence, we may write $f'_t$ for all $t$ as \[f'_t = f^-_t f^+_t\] where $\mathrm{supp}(f^{+}_t) \subset (M \setminus L_{+}) \cap U$ and $\mathrm{supp}(f^{-}_t) \subset M \setminus L^{n-1}_{-}$ for all $t.$ In particular, as $U \cap V = \emptyset,$ we have that on $V,$ $f_t = f'_t = f^-_t,$ which in particular is {\em strict} there and \[d_{\al}(\tilde{f}_t,\tilde{f}_t^-) < \delta+C_+\] where $C_{\pm} = \sup |\widetilde{\theta}|_\al < \infty$ over all $\widetilde{\theta} \in \widetilde{\mathcal{G}}(M \setminus L_{\pm})$.
Furthermore, $f^-_t = \id$ on a neighbourhood $U' \subset U$ of $L^{n-1}_-.$ Set $\phi_t = f^-_t.$ 

Applying the argument of \cite{Nakamura21}, which readily extends to the relative case, we see that for every $\delta>0$ there exists an isotopy $\psi_t$ with $\psi_t = \id$ on $U',$ $\psi_1(L) = \phi_1(L),$ $\psi_t(L)$ is $\delta$ close in the $C^0$-metric to $\phi_t(L),$ and finally \[|\tilde{\psi}_1|_{\al} < C(V) + C(\phi_1(V))+\delta = 2 C(V) + \delta.\] Here we used the fact that $\phi_1$ is strict on $V.$  

Finally, for $\delta$ sufficiently small, $\phi_t^{-1} \psi_t(L_-) \cap L_+ = \emptyset$ for all $t,$ whence \[d_{\al}(\tilde{\phi}_1, \tilde{\psi}_1) = |\tilde{\phi}_1^{-1}\tilde{\psi}_1|_{\al} \leq C_++C_-.\] In total we estimate \[ |\tilde{f}_t|_{\al} \leq |\tilde{\phi}_1|_{\al}+\delta+C_+ \leq |\tilde{\psi}_1|_{\al} + \delta+ 2C_++C_- \leq 3\delta+C(V)+2C_++C_-,\] a bound independent of $t.$
Therefore the contact Hofer norm on $\widetilde{\mathcal{G}}$ is bounded along the one-parameter subgroup given by the Reeb flow of $\alpha.$
In particular $(M,\xi)$ is not orderable. 

\subsection{Proof of Theorem \ref{thm1stab}}\label{sec:proof subcrit}
The skeleton $L$ of a Weinstein domain $W$ is an exact Lagrangian complex. We consider it as part of the page $W$ of the standard open book on $M=\partial_{\infty}(W \times \C).$ As such, it is a Legendrian complex. Furthermore, as the monodromy of this open book is trivial, it follows immediately from e.g. \cite[Section 2.11]{vankoert} that there is a contact form $\alpha$ on $M$ that is periodic near $L.$ Furthermore, as $W$ is obtained by attachment of $N$ critical Weinstein $n$-handles to the subcritical subdomain $W'$ of $W,$ the co-cores of the handles provide disjoint Lagrangian disks $D_i,$ $1 \leq i \leq N,$ each one of which intersects transversely $L$ at precisely one point. Moreover the cores of the handles correspond to the components $L^{n}_i,$ $1 \leq i \leq N,$ of $L^n = L \setminus L^{n-1},$ and $D_i \cap L^n_i = D_i \cap L.$ It follows from \cite[Proposition 2.9]{Casals19} that $L$ is a loose Legendrian complex in $M.$ Finally, by Theorem \ref{thm:perreeb}, we obtain that $M$ is not orderable.

\section{Examples and applications}

\subsection{The contact Hofer norm and the $\mathcal{C}^0$-topology}

Another consequence of Proposition \ref{prophandle} is that the contact Hofer norm can be bounded along $\mathcal{C}^0$ small isotopies.

\begin{cor}
Let $(M,\ker\alpha)$ be a closed contact manifold.
Then there exists a $\mathcal{C}^0$-neighbourhood $V$ of $\mathrm{id}$ in $\mathcal{G}$ such that $|\cdot|_{\alpha}$ is bounded on the path-connected component $V'$ of $V$ containing $\mathrm{id}$. Moreover, the Hofer norm is bounded on the image $\widetilde{V}' \to \widetilde{\mathcal{G}}.$
\end{cor}

Note that this statement is in striking contrast to the case of the group $\mathrm{Ham}(M,\om)$ Hamiltonian diffeomorphisms of symplectic manifolds. It is known, due originally to an argument of Ostrover, that there are autonomous Hamiltonian flows with arbitrarily small support with linearly growing Hofer norm in the universal cover $\widetilde{\mathrm{Ham}}(M,\om)$ of $\mathrm{Ham}(M,\om)$ for every closed symplectic manifold $(M,\om).$ This often produces linear growth in the group $\mathrm{Ham}(M,\om)$ itself, for instance whenever $(M,\om)$ is symplectically aspherical. We refer to \cite[Section 6.3.1]{PolterovichRosen} for details.

\begin{proof}
Let $(K,\theta)$ be an open book decomposition and $L_{\pm}$ be the skeletons of two different pages.
We can choose $V$ small enough such that for any $\phi\in V$ in the path-connected component of the identity, there exists a path $\phi_t$ in $V$ with $\phi_0=\mathrm{id}$ and $\phi_1=\phi$ such that $\phi_t(L_-)\cap L_+=\emptyset$ for all $t\in[0,1]$.
By Lemma \ref{lemfragmentation} we can decompose $\phi_t=f^-_tf^+_t$, where $f_t^{\pm}$ are compactly supported in $M\setminus L_{\pm}$.
Then Proposition \ref{prophandle} implies that the contact Hofer norm of any such $\phi$ is bounded from above independent of $\phi$ and in fact the same is true for the homotopy class of the path $\{\phi_t\}$ with fixed endpoints.
\end{proof} 

Note that it is not known if for every $\mathcal{C}^0$-small contactomorphism there exists a $\mathcal{C}^0$ small path connecting it to the identity. Therefore we pose the following questions.

\begin{qu}
Is $\mathcal{G}$ with the $\mathcal{C}^0$-topology locally path-connected?
\end{qu}

Note that in the symplectic case, it is known that the group $\mathrm{Ham}_c(M,\om)$ of Hamiltonian diffeomorphisms with compact supports is locally path connected in the $\mathcal{C}^0$ topology for $(M,\om)$ the standard symplectic open ball and compact symplectic surfaces, see \cite{Seyfaddini13}.

\begin{qu}
Is it true that the contact Hofer norm is bounded for $\mathcal{C}^0$-small contactomorphisms (not necessarily connected to $\mathrm{id}$ with a $\mathcal{C}^0$-small path)?
\end{qu}

\subsection{Translated points}\label{sec:transl}

Let $\alpha$ be a contact form on a closed manifold $M$ and $\phi\in \mathcal{G}$.
Following Sandon \cite{Sandon12}, we call a point $x\in M$ a \textbf{translated point} of $\phi$, if $(\phi^{\ast}\alpha)_x=\alpha_x$ and $\phi(x)=\phi_t^{\alpha}(x)$ for some $t\in \R$.

A first relation between translated points and the contact Hofer norm was found in \cite[Theorem B]{Egor}, where it was proved that in the case of Liouville fillable contact manifolds, any contactomorphism whose distance to the Reeb flow is sufficiently small, admits a translated point. In \cite{Cant24} it is shown that on the standard contact sphere this result is sharp, i.e. that there exist contactomorphisms without translated points and contact Hofer norm arbitrarily close to the bound in \cite{Egor}.

Note that the tools for constructing contactomorphisms without translated points introduced by Cant \cite{Cant22} are very similar to the tools used in the proof of Theorem \ref{thm1}. In particular, the flow constructed in Lemma \ref{lem_ob_contraction} has similar focal properties as the contactomorphisms used in \cite{Cant22} and \cite{Cant24}.

In the spirit of Theorem \ref{thm1} one can show:

\begin{theorem}\label{thm:notr}
 Let $(M,\xi)$ be a closed contact manifold that admits an open book decomposition $(K,\theta)$ whose pages are $1$-stabilizations of a Liouville manifold with respect to some ideal Giroux form $\beta$.
 Let $\alpha$ be a contact form whose Reeb flow coincides with the $S^1$-fibration $\theta$ in a neighbourhood around the skeletons of the pages.
 Then there exists a contactomorphism without translated points for $\alpha$.
\end{theorem}

\begin{proof}
We follow the idea of the proof in \cite{Cant22}.
Let $L_{\pm}$ be the skeletons of two different pages of the open book decomposition $(K,\theta)$.
First note that the flow $\phi_t$ of the vector field $X$ from Lemma \ref{lem_ob_contraction} has a similar focal property as in \cite[Definition 5]{Cant22} with respect to $(L_-,L_+)$.
Instead of contracting to a single point, we have for all $z\in M\setminus L_-$ that 
$$\lim\limits_{t\rightarrow\infty}\phi_t(z)\in L_+.$$
Note that at $L_{\pm}$ the scaling factor of $\phi_t$ with respect to $\beta$ and hence with respect to $\alpha$ is given by $e^{\pm t}$.
Analogously to \cite[Lemma 2]{Cant22} one can show that the scaling factor $1$-set $\Sigma_t:=\{z\in M|\phi_t^{\ast}\alpha=\alpha\}$ satisfies that
$$\lim_{t \to \infty} \mathrm{dist}(\Sigma_t, L_-)+\mathrm{dist}(\phi_t(\Sigma_t), L_+)=0.$$
As in \cite[Lemma 4]{Cant22} one can easily check that given a contactomorphism $\sigma$, the flow $\sigma\phi_t\sigma^{-1}$ has the focal property for the sets $\sigma(L_{\pm})$.
Let $U$ be a neighbourhood of the skeletons such that the Reeb vector field of $\alpha$ is given by $\partial_\theta$ on $U$.
Since $L_-$ is the skeleton of a $1$-stabilization, there exists a contactomorphism $\sigma$ supported away from $L_+$ that displaces $L_-$ from $U$.
Then, analogously to \cite{Cant22}, for $T$ large the contactomorphism $\sigma\phi_T\sigma^{-1}$ does not have translated points for $\alpha$.
\end{proof}

The existence of translated points for every contactomorphism and every contact form, has been proven in several cases including hypertight contact manifolds \cite{Albers15, Meiwes18}, certain lens spaces \cite{Granja21, Allais22} and certain prequantization spaces \cite{Tervil21, Bae24}.
However, it is likely that all of these examples are orderable, see e.g. \cite{Albers15, Borman152}.
Given the considerations above it seems that the non-existence of translated points is related to non-orderability of the underlying contact manifold.
Therefore we post the following question.

\begin{qu}
Is it true that a cooriented closed contact manifold $(M,\xi)$ is non-orderable if and only if there exists a contact form $\alpha$ with $\ker\alpha=\xi$ and a contactomorphism $\phi$ having no translated points for $\alpha$?
\end{qu}

As a first step it would be for instance interesting to consider the case $T^n\times S^{n+1}$ which is non-orderable by Corollary \ref{corst} but on which, to our knowledge, the general existence of translated points is not known. Similarly, it would be interesting to know if abstract open books with flexible pages admit contactomorphisms without translated points.
 
\subsection{Orderability and the Weinstein conjecture}

In view of \cite[Theorem 1.1]{Albers15}, on a non-orderable closed contact manifold, any contact form has a contractible closed Reeb orbit.
In particular the Weinstein conjecture holds for such manifolds.
Even though it seems to be very hard to obtain examples from Theorem \ref{thm1} and Theorem \ref{thm2} for which the Weinstein conjecture is not known to hold, we obtain the following criterion for the existence of contractible closed Reeb orbits in terms of open book decompositions.

\begin{cor}\label{Cor_Weinstein}
 Let $(M,\xi)$ be a closed contact manifold that satisfies the hypothesis of Theorem \ref{thm1} or Theorem \ref{thm2}.
 Then every contact form for $\xi$ has a contractible closed Reeb orbit.
\end{cor}

Corollary \ref{Cor_Weinstein} might be for instance applied to the case of prequantization spaces, giving a criterion in terms of open book decompositions for the fibre to be torsion.

\subsection{Prequantization spaces and subcritical polarizations}\label{sec:prequ}

Given a closed integral symplectic manifold $(Y,\omega)$ (meaning that $[\omega]$ lies in the image of the natural map $H^2(Y,\Z)\rightarrow H^2(Y,\R)$), there exists a unique complex line bundle over $Y$ with Euler class $[\omega]$.
The associated principal $S^1$-bundle $P\rightarrow Y$, called \textbf{prequatization bundle} or Boothby-Wang bundle, carries a contact form $\alpha$ whose Reeb flow is given by the $S^1$-action, see e.g. \cite[Section 7.2]{Geiges}.

Given an integral symplectic manifold $(Y,\omega)$ with Donaldson hypersurface $\Sigma$ that is Poincaré dual to $[\omega]$, there exists an open book decomposition of $(P,\ker\alpha)$ whose page is given by $Y\setminus \Sigma$, see \cite{Chiang14}.
In the context of Kähler manifolds, the choice of such a hypersurface is also called a polarization \cite[Section 2]{BiranCieliebak01}.
Together with Theorem \ref{thm1} this proves Corollary \ref{cor:Kahler}. In fact, Theorem \ref{thm:perreeb} implies the following generalization of this corollary. 

\begin{cor}\label{cor:preq noloose}
Let $P$ be an orderable prequantization of a symplectic manifold. Then it does not admit an open book whose pages have loose skeleta.  
\end{cor}

Biran and Cieliebak provide several examples of Kähler manifolds that admit a \textbf{subcritical polarization} of degree $1$, i.e., Kähler manifolds with a Donaldson hypersurface Poincaré dual to the symplectic form and a complement that is a subcritical Weinstein manifold.
Examples include for instance the standard $\C P^{n}$ (with $\omega_0=\frac{1}{\pi}\omega_{FS}$) or products of $\C P^n$ with its closed algebraic submanifolds of complex dimension $m<n.$

A direct consequence of Corollary \ref{cor:preq noloose} is that symplectic manifolds whose prequantization space is orderable do not admit subcritical polarizations of degree $1$.
An easy class of example is given by $(\C P^n,{k}\omega_{0})$ for any natural number $k\geq 2$, whose prequantization spaces are orderable Lens spaces \cite{Granja21, Allais22}.
In particular it follows that $(\C P^n,\omega_{0})$ does not admit any subcritical polarizations of degree $k$ for $k\geq 2$.

Similarly, Theorem \ref{thm:notr} implies that every symplectic manifold whose prequantization satisfies Sandon's conjecture, namely every $\psi \in \mathrm{Cont}_0(P,\xi)$ has translated points, does not admit a subcritical polarization. For instance, by \cite{Tervil21}, this is the case for $(M,\omega)$ being any monotone toric manifold with $[\omega]$ integral and primitive, other than $(\C P^n, \omega_{0}),$ does not admit a subcritical polarization. This reproves and greatly generalizes a folklore fact attributed to Biran that $\C P^n \times \C P^n$ does not admit subcritical polarizations. One gets a similar conclusion for degree $k\geq 2$ polarizations of a large class of monotone (not necessarily toric) symplectic manifolds by \cite{Albers23, Bae24}. For other obstructions on the polarization {\em hypersurfaces} of subcritical polarizations see \cite{BiranJerby}.

In a similar spirit, we immediately see from Theorem \ref{thm:notr} that contact manifolds satisfying Sandon's conjecture, including the examples above as well as hypertight contact manifolds \cite{Albers15, Meiwes18} and certain lens spaces \cite{Granja21, Allais22}, do not admit open books with subcritical pages.

\subsection{Non-displacement of pre-Lagrangians}

Recall that a submanifold $K$ in a contact manifold $(M,\xi)$ is called \textbf{pre-Lagrangian} if it has a Lagrangian lift to the symplectisation $SM$, i.e., if there exists a Lagrangian $L$ in $SM$ such that the projection of $SM$ onto $M$ maps $L$ diffeomorphically onto $K$. 
A first connection between non-displacement results for Legandrians and pre-Lagrangians and orderability was already found by Eliashberg and Polterovich \cite{Eliashberg00}.
In particular \cite[Theorem 2.3.A]{Eliashberg00} proves that a contact manifold is orderable if there exists a pre-Lagrangian or Legendrian $A$ whose stabilization can not be displaced from the stabilization of a given pre-Lagrangian $K$ in $M\times T^{\ast}S^1$ (the pair $(K,A)$ has the stable intersection property).
A further connections between the displacement of pre-Lagrangian tori and orderability are discussed for instance in \cite{Borman152, Marincovic16} in the setting of toric contact manifolds.

A direct consequence of Theorem \ref{thm1} is that orderable contact manifolds only have Weinstein open book decompositions with critical pages.
As pointed out in the proof of corollary \ref{corst}, in the case of an open book with critical Weinstein page it happens in many cases, that the image of a skeleton under a Reeb flow is a pre-Lagrangian submanifold.
For instance using the standard open book of $\R P^{2n+1}$ with page given by $T^{\ast}\R P^n$, we recover the non displacement of the Legendrian $\R P^n$ from its image under the Reeb flow described in \cite[Corollary 1.24]{Borman152}.

In this sense, Theorem \ref{thm2} can be seen as a weak converse direction to the implication in \cite[Theorem 2.3.A]{Eliashberg00}.

\begin{qu}
  Is a closed contact manifold orderable if and only if it contains a pair $(K,A)$ with the stable intersection property?
\end{qu}

\subsection{An explicit construction for positive loops}\label{sec:explicit}

It is a natural question to ask whether our methods can be used to provide an explicit construction for contractible positive loops. For example, let us focus on Theorem \ref{thm1}\ref{thm1:subcrit} and Theorem \ref{thm2}.

The flow from Lemma \ref{lem_ob_contraction} can be used in certain cases to shorten the contact Hofer length of a path supported in the complement of a skeleton of a page of some open book.
Together with Lemma \ref{lem_minimum} this observation can be used to obtain contractible positive loops from a Reeb flow with certain properties.

Let $(K_1,\theta_1)$ and $(K_2,\theta_2)$ be two (possibly different) open book decompositions supporting $(M,\xi)$.
Assume there exist a page $W_1$ of $(K_1,\theta_1)$ with skeleton $L_1$, a page $W_2$ of $(K_2,\theta_2)$ with skeleton $L_2$ and a Reeb flow $\phi_t^{\alpha}$ such that $\phi_t^{\alpha}(L_1)\cap L_2=\emptyset$.
Set $U_i:=M\setminus L_i$.
By Lemma \ref{lemfragmentation} we have
$$\phi_t^{\alpha}=f_t^1f_t^2,$$
where $f_t^i$ is compactly supported in $K_i$.
Lemma \ref{lem_ob_contraction} provides flows $\psi_t^i$ contracting $U_i$ onto the skeleton of some page of $(K_i,\theta_i)$.
As in the proof of Proposition \ref{prophandle}, the flows $\phi_t^i$ can be used to define flows $g_t^i$ compactly supported in $U_i$ that coincide on the support of $f_t^i$ with $\psi_t^i$.
For  $s\in [0,1]$ and fixed $T>0, C>0$ define the paths
$$ \phi_s^i:=g^i_{sT}(g^i_T)^{-1}f^i_{sC}g^i_T(g^i_{sT})^{-1}.$$
Then $\phi_s^i$ defines a path from the identity to $f_C^i$.
As pointed out in the proof of \cite[Theorem A]{Egor}, the triangle inequality and the naturality of the contact Hofer metric already hold on the level of the contact Hofer length $l_{\alpha}$.
In particular one has
$$l_{\alpha}( \phi_s^i)\leq 2l_{\alpha}( g^i_{sT})+l_{( g^i_{sT})^{\ast}\alpha}(f^i_{sC}).$$
As in the proof of Proposition \ref{prophandle}, there exists a constant $D_i$ such that for any $C$ one can find a sufficiently large $T$ with $l_{\alpha}( \phi_s^i)\leq D_i$.
Choose $C> D_1+D_2$.
Then the path $a_{s}:=\phi^{\alpha}_{sC}(\phi^2_{s})^{-1}(\phi^1_{s})^{-1}$ is a non-constant loop based at the identity with 
$$\int\limits_0^1\min\limits_M\alpha\left(\frac{d}{ds}a_s\right) ds>0.$$
As in the proof of Lemma \cite[Lemma 3.1]{Hedicke24} a positive contractible loop is given by $a_s\phi^{\alpha}_{\tau(s)}$, where $\tau\colon [0,1]\rightarrow \R$ is a suitable reparametrisation.

Note that the maps $g_t^i$ have similar properties as the map $b$ used in \cite[Remark 8.2]{Eliashberg06} to construct an explicit positive contractible loop on $S^3$ (see also \cite{Cant22}).
This raises the following questions:

\begin{qu}
 Do there exist Reeb flows and open book decompositions on $S^3$ (and more generally the contact boundary of a $2$-stabilization) such that the positive contractible loop constructed by the method above coincides with the loop from \cite{Eliashberg06}?  
\end{qu}

\begin{qu}
Are the positive contractible loops constructed from different Reeb flows and open book decompositions always homotopic via a homotopy consisting of positive contractible loops? See \cite[Section 6.3]{Eliashberg06} for related questions.
\end{qu}

\subsection{Asymptotically near-Reeb contractible loops}\label{sec:almostReeb}

An amusing fact due to J. Zhang and the second author is that the asymptotic contact Hofer norm of the Reeb flow is reflected in a certain measure of positive contractible loops. We state this relation here without proof, as we do not use it in the rest of the paper.

For a strictly positive contractible loop $\{\theta_t\}$ of contactomorphisms with contact Hamiltonian $h \in C^{\infty}([0,1] \times M,\R_{>0})$ set \[C(\theta) = \frac{\int_0^1 \max_M h_s\,ds}{\int_0^1 \min_M h_s\,ds}.\] Note that $C(\theta)\geq 1.$ Set \[ C(\alpha) = \inf_{\theta} C(\theta)\] with the infimum running over all positive loops $\theta.$ Let \[\mu(\alpha) = \lim_{t \to \infty} \frac{|\widetilde{\phi}^t_{R_{\alpha}}|_{\alpha}}{t}\] be the asymptotic contact Hofer growth of the Reeb flow in $\widetilde{\mathrm{Cont}}(M,\xi).$ 

\begin{prop}\label{prop:postcontr}
The following identity holds: \[ \mu(\alpha) = 1 - \frac{2}{C(\alpha)+1}.\]    
\end{prop}

In particular, for the manifolds $(M,\alpha)$ from the proof of Theorems \ref{thm1},\ref{thm2} and from Corollaries \ref{cor1stab},\ref{corst},\ref{corekp},\ref{cor:Kahler} where we have proven that $\mu(\alpha) = 0,$ we obtain that $C(\alpha) = 1,$ or in other words there exist positive contractible loops with $C(\theta)$ is arbitrarily close to $1.$ We can consider them to be ``close" to the Reeb flow in a certain quantitative sense. 

Note that for orderable manifolds $\mu(\alpha) = 1$ for all $\alpha,$ whence $C(\alpha)=+\infty.$ It would be interesting to study possible values of $C(\alpha)$ for non-orderable manifolds other than those mentioned above and see if intermediate values $C(\alpha) \in (0,\infty)$ could be achieved.

\bibliographystyle{siam}

\end{document}